\documentclass[a4paper,fleqn]{cas-dc}

\usepackage[numbers]{natbib}

\def\tsc#1{\csdef{#1}{\textsc{\lowercase{#1}}\xspace}}
\tsc{WGM}
\tsc{QE}
\tsc{EP}
\tsc{PMS}
\tsc{BEC}
\tsc{DE}

\newtheorem{assum}{Assumption}
\newtheorem{theorem}{Theorem}
\newtheorem{lemma}[theorem]{Lemma}
\newdefinition{definition}{Definition}
\newdefinition{remark}{Remark}
\newproof{proof}{Proof}

\allowdisplaybreaks

\newcommand*\diff{\mathop{}\!\mathrm{d}}%

\begin{document}
\let\WriteBookmarks\relax
\def\floatpagepagefraction{1}
\def\textpagefraction{.001}
\shortauthors{H. Lhachemi et~al.}

\title [mode = title]{Exponential Input-to-State Stabilization of a Class of Diagonal Boundary Control Systems with Delay Boundary Control}                      
\tnotemark[1]

\tnotetext[1]{This publication has emanated from research supported in part by a research grant from Science Foundation Ireland (SFI) under grant number 16/RC/3872 and is co-funded under the European Regional Development Fund and by I-Form industry partners.}


\author[1]{Hugo Lhachemi}[
                        orcid=0000-0002-2782-4208
                        ]
                        
\cormark[1]
\ead{hugo.lhachemi@ucd.ie}


\address[1]{School of Electrical and Electronic Engineering, University College Dublin, Dublin 4, D04 V1W8, Ireland}

\author[1,2]{Robert Shorten}[
                        orcid=0000-0002-9239-2499
                        ]
                        
\ead{robert.shorten@ucd.ie}

\address[2]{Dyson School of Design Engineering, Imperial College London, London, U.K.}

\author[3]{Christophe Prieur}[
                        orcid=0000-0002-4456-2019
                        ]

\ead{christophe.prieur@gipsa-lab.fr}

\address[3]{CNRS, Grenoble-INP, GIPSA-Lab, Universit{\'e} Grenoble Alpes, 38000 Grenoble, France}

\cortext[cor1]{Corresponding author}



\begin{abstract}
This paper deals with the exponential input-to-state stabilization with respect to boundary disturbances of a class of diagonal infinite-dimensional systems via delay boundary control. The considered input delays are uncertain and time-varying. The proposed control strategy consists of a constant-delay predictor feedback controller designed on a truncated finite-dimensional model capturing the unstable modes of the original infinite-dimensional system. We show that the resulting closed-loop system is exponentially input-to-state stable with fading memory of both additive boundary input perturbations and disturbances in the computation of the predictor feedback.
\end{abstract}




\maketitle

\section{Introduction}

Feedback stabilization of finite-dimensional systems in the presence of input delays has been a very active research topic during the past decades~\cite{artstein1982linear,richard2003time}. Motivated by the delay boundary control of Partial Differential Equations (PDEs), the opportunity of extending this topic to infinite-dimensional systems has recently attracted much attention~\cite{fridman2009exponential,solomon2015stability}. One of the early contributions on input delayed unstable PDEs, reported in~\cite{krstic2009control}, deals with a reaction-diffusion equation with a controller designed by resorting to the backstepping technique. More recently, the opportunity to use a predictor feedback for the stabilization of a reaction-diffusion equation was reported in~\cite{prieur2018feedback}. The proposed control strategy, inspired by the early works~\cite{coron2004global,coron2006global,russell1978controllability} dealing with delay-free boundary feedback control, goes as follows. First, a finite-dimensional truncated model capturing the unstable modes of the infinite-dimensional system is obtained via spectral reduction. Then, using the Artstein transformation for handling the input delay, a predictor feedback is designed to stabilize the truncated model. Finally, the stability of the closed-loop infinite-dimensional system is assessed via a Lyapunov-based argument. This strategy was reused in~\cite{guzman2019stabilization} for the delay boundary feedback stabilization of a linear Kuramoto-Sivashinsky equation. This was then generalized to the boundary feedback stabilization of a class of diagonal infinite-dimensional systems with delay boundary control for either a constant~\cite{lhachemi2019feedback,lhachemi2019control} or a time-varying~\cite{lhachemi2019lmi} input delay. 

In this paper, we investigate the exponential input-to-state stabilization with respect to boundary disturbances of a class of diagonal infinite-dimensional systems via delay boundary control. In this setting, the considered input delay is uncertain and time-varying. The main motivation in achieving an input-to-state stabilization of the closed-loop system relies in the fact that the Input-to-State Stability (ISS) property, originally introduced by Sontag in~\cite{sontag1989smooth}, is one of the main tools for assessing the robustness of a system with respect to boundary disturbances. This property also plays a key role in the establishment of small gain conditions for the stability of interconnected systems~\cite{karafyllis2019preview}. Although the study of ISS properties of finite-dimensional systems has been intensively studied during the last three decades, its extension to infinite-dimensional systems, and in particular with respect to boundary disturbances, is more recent~\cite{argomedo2012d,jacob2018infinite,jacob2018continuity,karafyllis2016iss,karafyllis2017iss,karafyllis2019preview,lhachemi2019input,lhachemi2018iss,mironchenko2019monotonicity,mironchenko2016restatements,mironchenko2017characterizations,zheng2018giorgi,zheng2017input}. Moreover, most of these results deal with the establishment of ISS properties for open-loop stable distributed parameter systems. The literature regarding the input-to-state stabilization of open-loop unstable infinite-dimensional systems is less developed.

In the context of recent efforts about the establishment of ISS properties w.r.t. exogenous disturbances for predictor feedback control of finite-dimensional systems~\cite{cai2017input,selivanov2016observer}, the present paper extends the results reported in~\cite{lhachemi2019feedback,lhachemi2019lmi} regarding the use of a constant-delay predictor feedback for the delayed boundary stabilization of a class of diagonal infinite-dimensional systems. The validity of such an approach was first assessed in~\cite{lhachemi2019feedback} for a constant, and known, input delay and then in~\cite{lhachemi2019lmi} for an unknown and time-varying input delay via Lyapunov-based arguments. While such an approach allows the derivation of an ISS estimate with respect to distributed disturbances~\cite{lhachemi2019feedback}, it fails in the establishement of an ISS estimate, in strict form\footnote{More precisely, this approach only allows the derivation of an ISS estimate with respect to both the boundary perturbation and its time derivative, but not an ISS estimate in strict form, i.e., with respect to the only magnitude of the boundary perturbation.}, with respect to boundary disturbances. It is worth noting that this increased difficulty regarding the establishment of ISS estimates w.r.t. boundary disturbances comparing to distributed ones seems to be a global trend for infinite-dimensional systems~\cite{mironchenko2019input}. In this paper, under the assumption of a sector condition on the eigenvalues corresponding to the modes which are not captured by the truncated model used for the design of the predictor feedback, we show that the resulting infinite-dimensional closed-loop system is exponentially ISS with fading memory~\cite{karafyllis2019preview} of the boundary disturbances for small variations of the time-varying delay around its nominal value. The adopted approach relies first on the extension of a small gain argument reported in~\cite{karafyllis2013delay} in order to establish the ISS property of the closed-loop truncated model, and then on the method reported in~\cite{lhachemi2018iss} for the establishment of ISS estimates with respect to boundary disturbances for diagonal infinite-dimensional systems.

This paper is organized as follows. The investigated control problem, the proposed control strategy, and the main result of this paper are introduced in Section~\ref{sec: problem setting and main results}. In Section~\ref{sec: exponential ISS of the truncated model} is reported the stability analysis of the finite-dimensional truncated model. Then, the proof of the main result of this paper, namely the ISS property of the resulting closed-loop infinite-dimensional system, is presented in Section~\ref{sec: exponential ISS of the infinite-dimensional system}. The relaxation of the assumed regularity assumptions for the boundary disturbances is discussed in Section~\ref{sec: exponential ISS for relaxed initial conditions and perturbations}. Finally, concluding remarks are formulated in Section~\ref{sec: conclusion}.

\section{Problem setting and main result}\label{sec: problem setting and main results}
The sets of non-negative integers, positive integers, real, non-negative real, positive real, and complex numbers are denoted by $\mathbb{N}$, $\mathbb{N}^*$, $\mathbb{R}$, $\mathbb{R}_+$, $\mathbb{R}_+^*$, and $\mathbb{C}$, respectively. Throughout the paper, the field $\mathbb{K}$ is either $\mathbb{R}$ or $\mathbb{C}$. All the finite-dimensional spaces $\mathbb{K}^p$ are endowed with the usual euclidean inner product $\left<x,y\right> = x^* y$ and the associated 2-norm $\Vert x \Vert = \sqrt{\left<x,x\right>} = \sqrt{x^* x}$. For any matrix $M \in \mathbb{K}^{p \times q}$, $\Vert M \Vert$ stands for the induced norm of $M$ associated with the above 2-norms. For any $t_0>0$, we say that $\varphi \in \mathcal{C}^0(\mathbb{R};\mathbb{R})$ is a \textit{transition signal over $[0,t_0]$} if $0 \leq \varphi \leq 1$, $\left. \varphi \right\vert_{(-\infty,0]} = 0$, and $\left. \varphi \right\vert_{[t_0,+\infty)} = 1$.

\subsection{Preliminary definitions}

Throughout the paper, $(\mathcal{H},\left< \cdot , \cdot \right>_\mathcal{H})$ denotes a separable Hilbert space over the field $\mathbb{K}$. 

\begin{definition}[Boundary control system~\cite{Curtain2012}]
Consider the abstract system taking the form:
\begin{equation}\label{def: boundary control system}
\left\{\begin{split}
\dfrac{\mathrm{d} X}{\mathrm{d} t}(t) & = \mathcal{A} X(t) , & t \geq 0 \\
\mathcal{B} X (t) & = v(t) , & t \geq 0 \\
X(0) & = X_0 
\end{split}\right.
\end{equation}
with $\mathcal{A} : D(\mathcal{A}) \subset \mathcal{H} \rightarrow \mathcal{H}$ an (unbounded) operator, $\mathcal{B} : D(\mathcal{B}) \subset \mathcal{H} \rightarrow \mathbb{K}^m$ with $D(\mathcal{A}) \subset D(\mathcal{B})$ the boundary operator, $v : \mathbb{R}_+ \rightarrow \mathbb{K}^m$ a boundary input, and $X_0 \in \mathcal{H}$ an initial condition. We say that $(\mathcal{A},\mathcal{B})$ is a boundary control system if:
\begin{enumerate}
\item the disturbance-free operator $\mathcal{A}_0$, defined on the domain $D(\mathcal{A}_0) \triangleq D(\mathcal{A}) \cap \mathrm{ker}(\mathcal{B})$ by $\mathcal{A}_0 \triangleq \left.\mathcal{A}\right|_{D(\mathcal{A}_0)}$, is the generator of a $C_0$-semigroup $S$ on $\mathcal{H}$;
\item there exists a bounded operator $B \in \mathcal{L}(\mathbb{K}^m,\mathcal{H})$, called a lifting operator, such that $\mathrm{R}(B) \subset D(\mathcal{A})$, $\mathcal{A}B \in \mathcal{L}(\mathbb{K}^m,\mathcal{H})$ (i.e., is a bounded operator), and $\mathcal{B}B = I_{\mathbb{K}^m}$.
\end{enumerate}
\end{definition}

\begin{definition}[Riesz spectral operator~\cite{Curtain2012}]
Let $\mathcal{A}_0 : D(\mathcal{A}_0) \allowbreak \subset \mathcal{H} \rightarrow \mathcal{H}$ be a linear and closed operator with simple eigenvalues $\lambda_n$ and corresponding eigenvectors $\phi_n \in D(\mathcal{A}_0)$, $n \in \mathbb{N}^*$. $\mathcal{A}_0$ is a Riesz-spectral operator if
\begin{enumerate}
\item $\left\{ \phi_n , \; n \in \mathbb{N}^* \right\}$ is a Riesz basis~\cite{christensen2016introduction}:
\begin{enumerate}
\item $\left\{ \phi_n , \; n \in \mathbb{N}^* \right\}$ is maximal, i.e., $\overline{ \underset{n\in\mathbb{N}^*}{\mathrm{span}_\mathbb{K}} \;\phi_n } = \mathcal{H}$;
\item there exist constants $m_R, M_R \in \mathbb{R}_+^*$ such that, for all $N \in \mathbb{N}^*$ and all $\alpha_1 , \ldots , \alpha_N \in \mathbb{K}$,
\end{enumerate}
\begin{equation}\label{eq: Riesz basis - inequality}
m_R \sum\limits_{n=1}^{N} \vert \alpha_n \vert^2
\leq
\left\Vert \sum\limits_{n=1}^{N} \alpha_n \phi_n \right\Vert_\mathcal{H}^2
\leq
M_R \sum\limits_{n=1}^{N} \vert \alpha_n \vert^2 ;
\end{equation}
\item the closure of $\{ \lambda_n , \; n \in \mathbb{N}^* \}$ is totally disconnected, i.e. for any distinct $a,b \in \overline{ \{ \lambda_n , \; n \in \mathbb{N}^* \} }$, we have $[a,b] \triangleq \{x a + (1-x) b \,:\, x\in[0,1]\} \not\subset \overline{ \{ \lambda_n , \; n \in \mathbb{N}^* \} }$.
\end{enumerate}
\end{definition}

\begin{remark}
Let $\left\{ \psi_n , \; n \in \mathbb{N}^* \right\}$ be the biorthogonal sequence associated with $\left\{ \phi_n , \; n \in \mathbb{N}^* \right\}$, i.e., $\left< \phi_n , \psi_m \right>_\mathcal{H} = \delta_{n,m}$. Then $\psi_n$ is an eigenvector of the adjoint operator $\mathcal{A}_0^*$ associated with $\overline{\lambda_n}$. Moreover, the following series expansion holds:
\begin{equation}\label{eq: series expansion Riesz basis}
\forall z \in\mathcal{H} , \quad 
z = \sum\limits_{n \geq 1} \left< z , \psi_n \right>_\mathcal{H} \phi_n .
\end{equation}
\end{remark}

\subsection{Problem and proposed control strategy}\label{subsec: problem and control strategy}

Let $D_0 > 0$ and $\delta \in (0,D_0)$ be given. We consider the abstract boundary control system (\ref{def: boundary control system}) for which the boundary input $v$ takes the form:
\begin{equation}\label{eq: boundary input}
v(t) = u(t-D(t)) + d_1(t) 
\end{equation}
for all $t \geq 0$ with $d_1 : \mathbb{R}_+ \rightarrow \mathbb{K}^m$ a boundary disturbance, $u : [-D_0-\delta,+\infty) \rightarrow \mathbb{K}^m$ the boundary control with $\left. u \right\vert_{[-D_0-\delta,0]} = 0$, and $D:\mathbb{R}_+ \rightarrow [D_0-\delta,D_0+\delta]$ a time-varying delay.

\begin{assum}\label{assum: A1}
The disturbance-free operator $\mathcal{A}_0$ is a Riesz spectral operator.
\end{assum}

Then, the $C_0$-semigroup generated by $\mathcal{A}_0$ is given by
\begin{equation}\label{eq: expression C0-semigroup}
\forall z \in\mathcal{H} , \quad 
\forall t \geq 0 , \quad
S(t)z = \sum\limits_{n \geq 1} e^{\lambda_n t} \left< z , \psi_n \right>_\mathcal{H} \phi_n .
\end{equation}

\begin{assum}\label{assum: A2}
There exist $N_0 \in \mathbb{N}^*$ and $\alpha \in \mathbb{R}_+^*$ such that 
\begin{enumerate}
\item $\operatorname{Re} \lambda_n \leq - \alpha$ for all $n \geq N_0 + 1$ ;
\item $\xi \triangleq \sup\limits_{n \geq N_0 +1} \left\vert \dfrac{\lambda_n}{\operatorname{Re} \lambda_n} \right\vert < \infty$.
\end{enumerate}
\end{assum}

\begin{remark}
If the first point of Assumption~\ref{assum: A2} holds, the second point $\xi < \infty$ is equivalent to the existence of a constant $\beta > 0$ such that $\vert \operatorname{Im} \lambda_n \vert \leq \beta \vert \operatorname{Re} \lambda_n \vert$ for all $n \geq N_0 + 1$. 
\end{remark}

The boundary feedback stabilization problem of the considered system was solved in~\cite{lhachemi2019lmi} in the disturbance-free case by designing a constant-delay predictor feedback on a finite dimensional truncated model capturing the unstable modes of the infinite-dimensional system. In this paper, we go beyond the result reported in~\cite{lhachemi2019lmi} by considering the impact of boundary disturbances while relaxing the assumed regularity properties and compatibility conditions. Specifically, assuming that the control input\footnote{The construction of the control law must ensure this property.} $u$, the time-varying delay $D$, and the boundary disturbance $d_1$ are of class $\mathcal{C}^1$, then, for any given initial condition $X_0 \in \mathcal{H}$, we can introduce $X \in \mathcal{C}^0(\mathbb{R}_+;\mathcal{H})$ defined for all $t \geq 0$ by
\begin{align}
X(t) & = S(t) \{ X_0 - B v(0) \} + B v(t) \label{eq: def mild solution} \\ 
& \phantom{=}\; + \int_0^t S(t-s) \{ \mathcal{A}Bv(s) - B \dot{v}(s) \} \diff s \nonumber
\end{align}
as the unique mild solution of (\ref{def: boundary control system}), with control input $v$ given by (\ref{eq: boundary input}), associated with $(D,X_0,d_1)$. We introduce the series expansion $X(t) = \sum_{n \geq 1} c_n(t) \phi_n$ with $c_n(t) \triangleq \left< X(t) , \psi_n \right>_\mathcal{H}$ the coefficients of projection of the system trajectory $X(t)$ into the Riesz basis $\left\{ \phi_n , \; n \in \mathbb{N}^* \right\}$. The use of (\ref{eq: def mild solution}), combined with (\ref{eq: series expansion Riesz basis}) and (\ref{eq: expression C0-semigroup}), and an integration by parts, show that $c_n$ statisfies
\begin{align}
c_n(t) & = 
e^{\lambda_n t} c_n(0) \label{eq: coeff in Riesz basis integral form} \\
& \phantom{=}\; + \int_0^t e^{\lambda_n (t-\tau)} \left\{ \left< \mathcal{A} B v(\tau) , \psi_n \right>_\mathcal{H} - \lambda_n \left< B v(\tau) , \psi_n \right>_\mathcal{H}  \right\} \diff\tau \nonumber
\end{align}
for all $t \geq 0$. Thus $c_n \in \mathcal{C}^1(\mathbb{R}_+;\mathbb{K})$ and satisfies for all $t \geq 0$ the following ODE (see also~\cite{lhachemi2018iss}):
\begin{equation}\label{eq: coeff in Riesz basis ODE}
\dot{c}_n(t) 
= \lambda_n c_n(t) - \lambda_n \left< B v(t) , \psi_n \right>_\mathcal{H} + \left< \mathcal{A} B v(t) , \psi_n \right>_\mathcal{H} .
\end{equation}
Let $\mathcal{E} = (e_1,e_2,\ldots,e_m)$ be the canonical basis of $\mathbb{K}^m$. Then, introducing\footnote{Note that the quantity $b_{n,k}$ is independent of the specifically selected lifting operator $B$ associated with $(\mathcal{A},\mathcal{B})$, see~\cite{lhachemi2019feedback}.} $b_{n,k} \triangleq - \lambda_n \left< B e_k , \psi_n \right>_\mathcal{H} + \left< \mathcal{A} B e_k , \psi_n \right>_\mathcal{H}$, we obtain that
\begin{subequations}\label{eq: ODE satisfies by Y}
\begin{align}
\dot{Y}(t) 
& = A_{N_0} Y(t) + B_{N_0} v(t) \nonumber \\
& = A_{N_0} Y(t) + B_{N_0} \{ u(t-D(t)) + d_1(t) \} , \\
Y(0) & = Y_0 ,
\end{align}
\end{subequations}
with 
\begin{equation}\label{eq: definition Y}
Y(t) = 
\begin{bmatrix}
c_1(t) & \ldots & c_{N_0}(t)
\end{bmatrix}^\top
\in \mathbb{K}^{N_0} ,
\end{equation}
the matrices $A_{N_0} = \mathrm{diag}(\lambda_1,\ldots,\lambda_{N_0}) \in \mathbb{K}^{N_0 \times N_0}$ and $B_{N_0} = (b_{n,k})_{1 \leq n \leq N_0 , 1 \leq k \leq m} \in \mathbb{K}^{N_0 \times m}$, and the initial condition
\begin{equation*}
Y_0 = 
\begin{bmatrix} 
\left< X_0 , \psi_1 \right>_\mathcal{H} & \ldots & \left< X_0 , \psi_{N_0} \right>_\mathcal{H} \end{bmatrix}
^\top
\in \mathbb{K}^{N_0} .
\end{equation*}

\begin{assum}\label{assum: A3}
$(A_{N_0},B_{N_0})$ is stabilizable.
\end{assum}

Under Assumption~\ref{assum: A3}, one can design a predictor feedback achieving the stabilization of the truncated model (\ref{eq: ODE satisfies by Y}). Then, following~\cite{lhachemi2019lmi}, such a predictor feedback can be successfully applied to the original infinite-dimensional system. Specifically, let $t_0,D_0 > 0$ and $\delta \in (0,D_0)$ be given. We consider a given transition signal\footnote{See the notation section at the beginning of Section~\ref{sec: problem setting and main results}.} $\varphi \in \mathcal{C}^1(\mathbb{R};\mathbb{R})$ over $[0,t_0]$. We assume that $D \in \mathcal{C}^1(\mathbb{R}_+;\mathbb{R})$ with $\vert D - D_0 \vert \leq \delta$. The closed-loop system dynamics takes the following form:
\begin{subequations}\label{def: boundary control system - closed-loop} 
\begin{align}
\dfrac{\mathrm{d} X}{\mathrm{d} t}(t) & = \mathcal{A} X(t) , \label{def: boundary control system - closed-loop - diff eq} \\
\mathcal{B} X (t) & = v(t) = u(t-D(t)) + d_1(t) , \label{def: boundary control system - closed-loop - bound cond} \\
u(t) & = \varphi(t) \Bigg\{  K Y(t)  + d_2(t) \label{def: boundary control system - closed-loop - cont input} \\
& \phantom{=}\; \hspace{1cm} + K \int_{\max(t-D_0,0)}^{t} e^{(t-s-D_0)A_{N_0}} B_{N_0} u(s) \diff s \Bigg\} , \nonumber \\
X(0) & = X_0 \label{def: boundary control system - closed-loop - ini cond}
\end{align}
\end{subequations}
for any $t \geq 0$. The adopted control strategy takes the form of a state-feedback in which the signal $Y(t)$ is computed based on the knowledge of the state $X(t)$ via (\ref{eq: definition Y}).  The feedback gain $K \in \mathbb{K}^{m \times N_0}$ is selected such that the matrix $A_\mathrm{cl} \triangleq A_{N_0} + e^{-D_0 A_{N_0}} B_{N_0} K$ is Hurwitz. Functions $d_1 , d_2 : \mathbb{R}_+ \rightarrow \mathbb{K}^{m}$ represent boundary disturbances.

\begin{remark}
Examples of systems covered by Assumptions~\ref{assum: A1}-\ref{assum: A3} and thus for which the proposed control strategy applies include reaction-diffusion equations~\cite{lhachemi2019boundary1,prieur2018feedback}, linear Kuramoto-Sivashinsky equation~\cite{guzman2019stabilization}, and certain damped flexible string or beam models~\cite[Ex.~2.23, p.~91]{Curtain2012}\cite{lhachemi2019boundary2}. For this type of system, the objective of the present paper is to establish a qualitative behavior regarding the closed-loop system dynamics (\ref{def: boundary control system - closed-loop}), namely an ISS property with respect to boundary disturbances $d_1$ and $d_2$.
\end{remark}

\begin{remark}
While disturbance $d_1$ represents an additive disturbance in the application of the delayed boundary control $u$, disturbance $d_2$ gathers uncertainties of either/both the output measurement $Y$ or/and the computation of the control law $u$ that is solution of a ``fixed point implicit equality'' involving an integral term~\cite{bresch2018new}. The existence and uniqueness of solutions for such an implicit equation has been assessed in~\cite{bresch2018new} in the case $\varphi = 1$. The proofs reported therein directly extend to the configuration studied in this paper by noting that $\varphi$ is a continuous function with $0 \leq \varphi \leq 1$. Moreover, as $Y$ is solution of the ODE (\ref{eq: ODE satisfies by Y}), it can be shown that the closed-loop dynamics (\ref{def: boundary control system - closed-loop}) with $Y$ given by (\ref{eq: definition Y}) is actually equivalent to the dynamics (\ref{def: boundary control system - closed-loop}) with $Y$ explicitly given by 
\begin{align*}
Y(t) & = e^{A_{N_0}t}Y_0 + \int_0^t e^{A_{N_0}(t-\tau)} B_{N_0}\left\{ u(\tau-D(\tau)) + d_1(\tau) \right\} \diff\tau .
\end{align*}
Note however that this second form is not convenient for practical implementation as it requires the knowledge of the disturbance $d_1$ in real-time.
\end{remark}

\subsection{Well-posedness in terms of mild solutions} 

In the first part of this paper, we  consider the following concept of mild solutions for the closed-loop system dynamics. 

\begin{definition}\label{def: mild solution closed-loop system}
Let $(\mathcal{A},\mathcal{B})$ be an abstract boundary control system such that Assumption~\ref{assum: A1} holds. Let $t_0,D_0 > 0$, $\delta \in (0,D_0)$, a transition signal $\varphi \in \mathcal{C}^1(\mathbb{R};\mathbb{R})$ over $[0,t_0]$, and $K \in \mathbb{K}^{m \times N_0}$ be arbitrary. For a time-varying delay $D \in \mathcal{C}^1(\mathbb{R}_+;\mathbb{R})$ with $\vert D - D_0 \vert \leq \delta$, an initial condition $X_0 \in \mathcal{H}$, and boundary perturbations $d_1,d_2\in\mathcal{C}^1(\mathbb{R}_+;\mathbb{K}^m)$, we say that $(X,u) \in \mathcal{C}^0(\mathbb{R}_+;\mathcal{H}) \times \mathcal{C}^1([-D_0-\delta,+\infty);\mathbb{K}^m)$ is a mild solution of (\ref{def: boundary control system - closed-loop}) associated with $(D,X_0,d_1,d_2)$ if 1) (\ref{eq: def mild solution}) holds for all $t \geq 0$ with $v$ given by (\ref{eq: boundary input}); 2) $u$ satisfies (\ref{def: boundary control system - closed-loop - cont input}) for all $t \geq -D_0-\delta$ with $Y$ defined by (\ref{eq: definition Y}).
\end{definition}

The following lemma, whose proof is placed in Appendix~\ref{annex: proof lemma well-posedness}, assesses the well-posedness of the closed-loop system (\ref{def: boundary control system - closed-loop}) in terms of mild solutions.

\begin{lemma}\label{lemma: well-posedness closed-loop system}
For any $D \in \mathcal{C}^1(\mathbb{R}_+;\mathbb{R})$ with $\vert D - D_0 \vert \leq \delta$, $X_0 \in \mathcal{H}$, and $d_1,d_2 \in \mathcal{C}^1(\mathbb{R}_+;\mathbb{K}^m)$, the closed-loop system (\ref{def: boundary control system - closed-loop}) admits a unique mild solution $(X,u) \in \mathcal{C}^0(\mathbb{R}_+;\mathcal{H}) \times \mathcal{C}^1([-D_0-\delta,+\infty);\mathbb{K}^m)$ associated with $(D,X_0,d_1,d_2)$.
\end{lemma}

\begin{remark}
If we assume the stronger regularity assumptions $\varphi\in\mathcal{C}^2(\mathbb{R};\mathbb{R})$, $D \in \mathcal{C}^2(\mathbb{R}_+;\mathbb{R})$, $d_1,d_2 \in \mathcal{C}^2(\mathbb{R}_+;\mathbb{K}^m)$, as well as the compatibility condition $X_0 \in D(\mathcal{A})$ with $\mathcal{B}X_0 = d_1(0)$, it can be shown that the mild solution is actually a classical solution with a control input $u$ that is twice continuously differentiable.
\end{remark}

\subsection{Main stability result}

The stability of the closed-loop system (\ref{def: boundary control system - closed-loop}) in the disturbance free case (i.e., for $d_1 = d_2 = 0$) was assessed in~\cite{lhachemi2019feedback,lhachemi2019control} for a constant delay $D(t) = D_0$ and in~\cite{lhachemi2019lmi} for an uncertain and time-varying delay $D(t)$. The objective of this paper is to study the impact of the boundary disturbances $d_1$ and $d_2$ on the system trajectories. More precisely, we derive the following result.

\begin{theorem}\label{thm: ISS mild solutions}
Let $(\mathcal{A},\mathcal{B})$ be an abstract boundary control system such that Assumptions~\ref{assum: A1}, \ref{assum: A2}, and~\ref{assum: A3} hold. Let $\varphi \in \mathcal{C}^1(\mathbb{R};\mathbb{R})$ be a transition signal over $[0,t_0]$ for some $t_0 > 0$. Let $D_0 > 0$ and $K \in \mathbb{K}^{m \times N_0}$ be such that $A_\mathrm{cl} = A_{N_0} + e^{-D_0 A_{N_0}} B_{N_0} K$ is Hurwitz. Let $\delta \in (0,D_0)$ be such that\footnote{Such a $\delta > 0$ always exists by a continuity argument in $\delta = 0$.} 
\begin{equation}\label{eq: thm infinite dim system - small gain condition}
M_\lambda \Vert B_{N_0} K \Vert \left[ e^{\Vert A_\mathrm{cl} \Vert \delta} - e^{-\lambda \delta} \right] < \lambda ,
\end{equation}
where $\lambda > 0$ and $M_\lambda \geq 1$ are such that $\Vert e^{A_\mathrm{cl} t} \Vert \leq M_\lambda e^{-\lambda t}$ for all $t \geq 0$. Then, there exist $\kappa \in (0,\alpha)$ and $\overline{C}_i > 0$, $1 \leq i \leq 6$, such that, for any $D \in \mathcal{C}^1(\mathbb{R}_+;\mathbb{R})$ with $\vert D - D_0 \vert \leq \delta$, $X_0 \in \mathcal{H}$, and  $d_1,d_2 \in \mathcal{C}^1(\mathbb{R}_+;\mathbb{K}^m)$, the mild solution $(X,u) \in \mathcal{C}^0(\mathbb{R}_+;\mathcal{H}) \times \mathcal{C}^1([-D_0-\delta,+\infty);\mathbb{K}^m)$ of the closed-loop system (\ref{def: boundary control system - closed-loop}) associated with $(D,X_0,d_1,d_2)$ satisfies
\begin{align}
\Vert X(t) \Vert_\mathcal{H}
& \leq \overline{C}_1 e^{- \kappa t} \Vert X_0 \Vert_\mathcal{H} 
+ \overline{C}_2 \sup\limits_{\tau \in [0,t]} e^{-\kappa (t-\tau)} \Vert d_1(\tau) \Vert \label{eq: ISS estimate} \\
& \phantom{\leq}\, + \overline{C}_3 \sup\limits_{\tau \in [0,\max(t-(D_0-\delta),0)]} e^{-\kappa (t-\tau)} \Vert d_2(\tau) \Vert \nonumber 
\end{align}
and
\begin{align}
\Vert u(t) \Vert 
& \leq
\overline{C}_4 e^{- \kappa t} \Vert X_0 \Vert_\mathcal{H} 
+ \overline{C}_5 \sup\limits_{\tau \in [0,t]} e^{-\kappa (t-\tau)} \Vert d_1(\tau) \Vert \label{eq: ISS estimate command} \\
& \phantom{\leq}\, + \overline{C}_6 \sup\limits_{\tau \in [0,t]} e^{-\kappa (t-\tau)} \Vert d_2(\tau) \Vert . \nonumber
\end{align}
for all $t \geq 0$.
\end{theorem}

The next two sections are devoted to the proof of Theorem~\ref{thm: ISS mild solutions}. The extension of this result to continuous boundary disturbances $d_1,d_2$ is discussed in Section~\ref{sec: exponential ISS for relaxed initial conditions and perturbations}.

\begin{remark}
It is interesting to note that Theorem~\ref{thm: ISS mild solutions}, involving the sector condtion $\xi < +\infty$ (see Assumption~\ref{assum: A2}), does not introduce any constraint on the amplitude of variation of the time derivative $\dot{D}$ of the input delay $D$. This is in contrast with the result reported in~\cite{lhachemi2019lmi} for the disturbance-free case (i.e., $d_1=d_2=0$), which allows $\xi = +\infty$ but where the constant of the exponential stability property is a strictly increasing function, going to $+\infty$ at $+\infty$, of the supremum of $\vert \dot{D} \vert$. The occurence of a $\dot{D}$ term in the proof of the result reported in~\cite{lhachemi2019lmi} is due to the use of a Lyapunov-based argument. As discussed in the sequel of this paper, the assumption $\xi < +\infty$ allows a proof of Theorem~\ref{thm: ISS mild solutions} that does not rely on such a Lyapunov-based argument.
\end{remark}

\section{Exponential ISS of the truncated model}\label{sec: exponential ISS of the truncated model}

In this section, we study the ISS property of the finite dimensional truncated model. We refer the reader to~\cite{sontag2008input} for classical results about the establishment of ISS properties.

\subsection{Preliminary lemma}
We need the following preliminary lemma which is a disturbed version of the disturbance-free version ($p=0$) reported in~\cite[Th.~2.5]{karafyllis2013delay}.

\begin{lemma}\label{lem: prel lemma}
Let $A \in \mathbb{K}^{n \times n}$ be Hurwitz, $C \in \mathbb{K}^{n \times n}$, and $r > 0$. Let $\epsilon \in (0,r)$ be such that 
\begin{equation}\label{eq: prel lemma - small gain condition}
M_\lambda \Vert C \Vert \left[ e^{\Vert A \Vert \epsilon} - e^{-\lambda \epsilon} \right] < \lambda ,
\end{equation}
where $\lambda > 0$ and $M_\lambda \geq 1$ are such that $\Vert e^{At} \Vert \leq M_\lambda e^{-\lambda t}$ for all $t \geq 0$. Then, there exist $\sigma,N > 0$ and $M \geq 1$ such that, for any $d \in \mathcal{C}^0(\mathbb{R}_+;\mathbb{R})$ with $\vert d \vert \leq 1$, any $p,q \in \mathcal{C}^0(\mathbb{R}_+;\mathbb{K})$ with $\vert q \vert \leq 1$, and any $x_0 \in \mathcal{C}^0([-r-\epsilon,0];\mathbb{K}^n)$, the trajectory of 
\begin{subequations}\label{eq: prel lemma}
\begin{align}
\dot{x} (t) & = A x(t) + q(t) C \left[ x(t-r-\epsilon d(t)) - x(t-r) \right] + p(t) \label{eq: prel lemma - ODE} \\
x(\tau) & = x_0(\tau) , \quad -r-\epsilon \leq \tau \leq 0 \label{eq: prel lemma - ODE IC}
\end{align}
\end{subequations}
for $t \geq 0$ satisfies
\begin{equation}\label{eq: lemma - claimed estimate}
\Vert x(t) \Vert 
\leq
M e^{- \sigma t} \sup\limits_{\tau \in [ -r-\epsilon , 0 ]} \; \Vert x_0(\tau) \Vert 
+ N \sup\limits_{\tau \in [ 0 , t ]} \; e^{-\sigma (t-\tau)} \Vert p(\tau) \Vert 
\end{equation}
for all $t \geq 0$.
\end{lemma}

\begin{proof}
As the case $C = 0$ is straightforward, we assume in the sequel that $C \neq 0$. The first part of the proof follows the one in~\cite{karafyllis2013delay} while considering the impact of the disturbing term $p$. We define, for all $t \geq 0$, $v(t) = x(t-r-\epsilon d(t)) - x(t-r)$. Let $\sigma \in (0,\lambda)$, which will be specified in the sequel, be arbitrary. The proof is divided into 4 main steps.

\emph{Step 1: preliminary estimation of $\sup\limits_{\tau \in [ r+\epsilon , t ]} \; e^{\sigma \tau} \Vert v(\tau) \Vert$ by an upper estimate involving $\Vert x(\tau) \Vert$.}
As in~\cite{karafyllis2013delay}, we consider the cases $d(t) \leq 0$ and $d(t) \geq 0$ separately. In the case $d(t) \leq 0$, we have by direct integration of (\ref{eq: prel lemma - ODE}) that, for all $t \geq r$,
\begin{align*}
v(t)
& = \left[ e^{-\epsilon A d(t)} - I_n \right] x(t-r) \\
& \phantom{=}\, + \int_{t-r}^{t-r-\epsilon d(t)} e^{A(t-r-\epsilon d(t) - \tau)} \left[ q(\tau) C v(\tau) + p(\tau) \right] \diff\tau 
\end{align*}
because $t-r-\epsilon d(t) \geq t-r \geq 0$.
Noting that $\Vert e^{-\epsilon A d(t)} - I_n \Vert \leq e^{\Vert A \Vert \epsilon} - 1$,
\begin{align*}
& \left\Vert \int_{t-r}^{t-r-\epsilon d(t)} e^{A(t-r-\epsilon d(t) - \tau)} 
q(\tau) C v(\tau) \diff\tau \right\Vert \\
& \leq M_\lambda \Vert C \Vert \int_{t-r}^{t-r-\epsilon d(t)} e^{-\lambda (t-r-\epsilon d(t) - \tau)}  \left\Vert v(\tau) \right\Vert \diff\tau \\
& \leq M_\lambda \Vert C \Vert e^{-\lambda (t-r-\epsilon d(t))} \int_{t-r}^{t-r-\epsilon d(t)} e^{(\lambda-\sigma) \tau}  \times e^{\sigma \tau} \left\Vert v(\tau) \right\Vert \diff\tau \\
& \leq M_\lambda \Vert C \Vert e^{-\sigma (t-r)} e^{\sigma\epsilon d(t)} \dfrac{1-e^{(\lambda-\sigma)\epsilon d(t)}}{\lambda-\sigma} \sup\limits_{\tau \in [ t-r , t-r+\epsilon ]} \; e^{\sigma\tau} \Vert v(\tau) \Vert \\
& \leq M_\lambda \Vert C \Vert e^{-\sigma (t-r)} \dfrac{1-e^{-(\lambda-\sigma)\epsilon}}{\lambda-\sigma} \sup\limits_{\tau \in [ t-r , t-r+\epsilon ]} \; e^{\sigma\tau} \Vert v(\tau) \Vert ,
\end{align*}
where it has been used that $-1 \leq d(t) \leq 0$, and, similarly,
\begin{align*}
& \left\Vert \int_{t-r}^{t-r-\epsilon d(t)} e^{A(t-r-\epsilon d(t) - \tau)} 
p(\tau) \diff\tau \right\Vert \\
& \leq M_\lambda e^{-\sigma (t-r)} \dfrac{1-e^{-(\lambda-\sigma)\epsilon}}{\lambda-\sigma} \sup\limits_{\tau \in [ t-r , t-r+\epsilon ]} \; e^{\sigma\tau} \Vert p(\tau) \Vert ,
\end{align*}
we obtain that, for all $t \geq r$ such that $d(t) \leq 0$,
\begin{align}
e^{\sigma t} \Vert v(t) \Vert
& \leq \left[ e^{\Vert A \Vert \epsilon} - 1 \right] e^{\sigma r} \times e^{\sigma (t-r)} \Vert x(t-r) \Vert \label{eq: proof prel lemma - int 1} \\
& \phantom{\leq}\, + M_\lambda \Vert C \Vert e^{\sigma r} \dfrac{1-e^{-(\lambda-\sigma)\epsilon}}{\lambda-\sigma} \sup\limits_{\tau \in [ t-r , t-r+\epsilon ]} \; e^{\sigma\tau} \Vert v(\tau) \Vert \nonumber \\
& \phantom{\leq}\, + M_\lambda e^{\sigma r} \dfrac{1-e^{-(\lambda-\sigma)\epsilon}}{\lambda-\sigma} \sup\limits_{\tau \in [ t-r , t-r+\epsilon ]} \; e^{\sigma\tau} \Vert p(\tau) \Vert . \nonumber
\end{align}
Now, in the case $d(t) \geq 0$, we have by direct integration of (\ref{eq: prel lemma - ODE}) that, for all $t \geq r + \epsilon$,
\begin{align*}
v(t)
& = - \left[ e^{\epsilon A d(t)} - I_n \right] x(t-r-\epsilon d(t)) \\
& \phantom{=}\; - \int_{t-r-\epsilon d(t)}^{t-r} e^{A(t-r-\tau)} \left[ q(\tau) C v(\tau) + p(\tau) \right] \diff\tau 
\end{align*}
because $t-r \geq t-r-\epsilon d(t) \geq t-r-\epsilon \geq 0$. Then, 
we deduce that, for all $t \geq r+\epsilon$ such that $d(t) \geq 0$,
\begin{align}
& e^{\sigma t} \Vert v(t) \Vert \nonumber \\
& \leq \left[ e^{\Vert A \Vert \epsilon} - 1 \right] e^{\sigma (r+\epsilon)} \times e^{\sigma (t-r-\epsilon d(t))} \Vert x(t-r-\epsilon d(t)) \Vert \label{eq: proof prel lemma - int 2} \\
& \phantom{\leq}\, + M_\lambda \Vert C \Vert e^{\sigma r} \dfrac{1-e^{-(\lambda-\sigma)\epsilon}}{\lambda-\sigma} \sup\limits_{\tau \in [ t-r-\epsilon , t-r ]} \; e^{\sigma\tau} \Vert v(\tau) \Vert \nonumber \\
& \phantom{\leq}\, + M_\lambda e^{\sigma r} \dfrac{1-e^{-(\lambda-\sigma)\epsilon}}{\lambda-\sigma} \sup\limits_{\tau \in [ t-r-\epsilon , t-r ]} \; e^{\sigma\tau} \Vert p(\tau) \Vert . \nonumber 
\end{align}
Combining (\ref{eq: proof prel lemma - int 1}-\ref{eq: proof prel lemma - int 2}), we obtain that, for all $t \geq r+ \epsilon$,
\begin{align}
& \sup\limits_{\tau \in [ r+\epsilon , t ]} \; e^{\sigma \tau} \Vert v(\tau) \Vert \nonumber \\
& \leq \left[ e^{\Vert A \Vert \epsilon} - 1 \right] e^{\sigma (r+\epsilon)} \sup\limits_{\tau \in [ 0 , t-r ]} \; e^{\sigma \tau} \Vert x(\tau) \Vert \label{eq: proof prel lemma - int 3}\\
& \phantom{\leq}\; + M_\lambda \Vert C \Vert e^{\sigma r} \dfrac{1-e^{-(\lambda-\sigma)\epsilon}}{\lambda-\sigma} \sup\limits_{\tau \in [ 0 , t-r+\epsilon ]} \; e^{\sigma\tau} \Vert v(\tau) \Vert \nonumber \\
& \phantom{\leq}\; + M_\lambda e^{\sigma r} \dfrac{1-e^{-(\lambda-\sigma)\epsilon}}{\lambda-\sigma} \sup\limits_{\tau \in [ 0 , t-r+\epsilon ]} \; e^{\sigma\tau} \Vert p(\tau) \Vert . \nonumber
\end{align}

\emph{Step 2: preliminary estimation of $\sup\limits_{\tau \in [ 0 , t ]} \; e^{\sigma \tau} \Vert x(\tau) \Vert$ by an upper estimate involving involving $\Vert v(\tau) \Vert$.}
Integrating (\ref{eq: prel lemma - ODE}) over $[0,t]$, we obtain for all $t \geq 0$,
\begin{equation*}
x(t) = e^{At} x(0) + \int_{0}^{t} e^{A(t-\tau)} \left[ q(\tau) C v(\tau) + p(\tau) \right] \diff\tau .
\end{equation*}
As $x(0) = x_0(0)$, straightforward estimations show that, for all $t \geq 0$,
\begin{align}
\sup\limits_{\tau \in [ 0 , t ]} \; e^{\sigma \tau} \Vert x(\tau) \Vert
& \leq
M_\lambda \Vert x_0(0) \Vert
+ \dfrac{M_\lambda \Vert C \Vert}{\lambda - \sigma} \sup\limits_{\tau \in [ 0 , t ]} \; e^{\sigma\tau} \Vert v(\tau) \Vert \nonumber \\
& \phantom{\leq}\; + \dfrac{M_\lambda}{\lambda - \sigma} \sup\limits_{\tau \in [ 0 , t ]} \; e^{\sigma\tau} \Vert p(\tau) \Vert . \label{eq: proof prel lemma - int 4}
\end{align}

\emph{Step 3: estimation of $\sup\limits_{\tau \in [ 0 , t ]} \; e^{\sigma \tau} \Vert x(\tau) \Vert$ by an upper estimate involving only $\Vert x_0(\tau) \Vert$ and $\Vert p(\tau) \Vert$ via a small gain argument.}
From (\ref{eq: proof prel lemma - int 3}-\ref{eq: proof prel lemma - int 4}), we deduce that, for all $t \geq r+\epsilon$,
\begin{align*}
\sup\limits_{\tau \in [ r+\epsilon , t ]} \; e^{\sigma \tau} \Vert v(\tau) \Vert
& \leq M_\lambda e^{\sigma (r+\epsilon)} \left[ e^{\Vert A \Vert \epsilon} - 1 \right] \Vert x_0(0) \Vert \\
& \phantom{\leq}\, + \delta \sup\limits_{\tau \in [ 0 , t-r+\epsilon ]} \; e^{\sigma\tau} \Vert v(\tau) \Vert \\
& \phantom{\leq}\, + \dfrac{\delta}{\Vert C \Vert} \sup\limits_{\tau \in [ 0 , t-r+\epsilon ]} \; e^{\sigma\tau} \Vert p(\tau) \Vert ,
\end{align*}
where 
\begin{equation*}
\delta \triangleq \dfrac{M_\lambda \Vert C \Vert}{\lambda-\sigma} e^{\sigma r}
\left[ e^{\sigma \epsilon} ( e^{\Vert A \Vert \epsilon} - 1) + 1 - e^{-(\lambda-\sigma) \epsilon} \right] .
\end{equation*}
From the small gain assumption (\ref{eq: prel lemma - small gain condition}) and a continuity argument in $\sigma = 0$, we select $\sigma \in (0,\lambda)$ such that $\delta < 1$. Noting that the supremums appearing in the latter estimate are finite, we deduce that, for all $t \geq 0$,  
\begin{align*}
\sup\limits_{\tau \in [ 0 , t ]} \; e^{\sigma \tau} \Vert v(\tau) \Vert
& \leq \dfrac{M_\lambda e^{\sigma (r+\epsilon)} \left[ e^{\Vert A \Vert \epsilon} - 1 \right]}{1-\delta} \Vert x_0(0) \Vert \\
& \phantom{\leq}\; + \sup\limits_{\tau \in [ 0 , r + \epsilon ]} \; e^{\sigma\tau} \Vert v(\tau) \Vert \\
& \phantom{\leq}\; + \dfrac{\delta}{\Vert C \Vert (1-\delta)} \sup\limits_{\tau \in [ 0 , t ]} \; e^{\sigma\tau} \Vert p(\tau) \Vert . 
\end{align*}
Using (\ref{eq: proof prel lemma - int 4}), we obtain that, for all $t \geq 0$,  
\begin{align}
& \sup\limits_{\tau \in [ 0 , t ]} \; e^{\sigma \tau} \Vert x(\tau) \Vert \nonumber \\
& \leq
M_\lambda \left\{ 1 + \dfrac{M_\lambda \Vert C \Vert e^{\sigma (r+\epsilon)} \left[ e^{\Vert A \Vert \epsilon} - 1 \right]}{(1-\delta)(\lambda-\sigma)} \right\} \Vert x_0(0) \Vert \label{eq: proof prel lemma - int 5} \\
& \phantom{\leq}\, + \dfrac{M_\lambda \Vert C \Vert}{\lambda - \sigma} \sup\limits_{\tau \in [ 0 , r+\epsilon ]} \; e^{\sigma\tau} \Vert v(\tau) \Vert \nonumber \\
& \phantom{\leq}\, + \dfrac{M_\lambda}{(1-\delta)(\lambda-\sigma)} \sup\limits_{\tau \in [ 0 , t ]} \; e^{\sigma\tau} \Vert p(\tau) \Vert . \nonumber
\end{align}
It remains now to evaluate $\sup\limits_{\tau \in [ 0 , r+\epsilon ]} \; e^{\sigma\tau} \Vert v(\tau) \Vert$. To do so, we note from the definition of $v$ that
\begin{align}
& \sup\limits_{\tau \in [ 0 , r+\epsilon ]} \; e^{\sigma\tau} \Vert v(\tau) \Vert \nonumber \\
& \leq
2 e^{\sigma(r+\epsilon)} \left( 
\sup\limits_{\tau \in [ -r-\epsilon , 0 ]} \; \Vert x_0(\tau) \Vert 
+ \sup\limits_{\tau \in [ 0 , 2 \epsilon ]} \; \Vert x(\tau) \Vert 
\right) . \label{eq: proof prel lemma - int 6}
\end{align}
Based on (\ref{eq: prel lemma}), a standard application of Gr{\"o}nwall's inequality shows the existence of constants $\gamma_0,\gamma_1 > 0$, which only depend on $A,C,r,\epsilon$, such that
\begin{align}
\sup\limits_{\tau \in [ 0 , 2\epsilon ]} \; \Vert x(\tau) \Vert 
& \leq \gamma_0 \sup\limits_{\tau \in [ -r-\epsilon , 0 ]} \; \Vert x_0(\tau) \Vert + \gamma_1 \sup\limits_{\tau \in [ 0 , 2\epsilon ]} \;  e^{\sigma\tau}\Vert p(\tau) \Vert . \label{eq: proof prel lemma - int 7}
\end{align}
Combining (\ref{eq: proof prel lemma - int 5}-\ref{eq: proof prel lemma - int 7}), we deduce the existence of constants $M \geq 1$ and $N > 0$ such that, for all $t \geq 0$,
\begin{align}
& \sup\limits_{\tau \in [ 0 , t ]} \; e^{\sigma \tau} \Vert x(\tau) \Vert \nonumber \\
& \leq
M \sup\limits_{\tau \in [ -r-\epsilon , 0 ]} \; \Vert x_0(\tau) \Vert
+ N \sup\limits_{\tau \in [ 0 , \max(t,2
\epsilon) ]} \; e^{\sigma\tau} \Vert p(\tau) \Vert . \label{eq: proof prel lemma - int 8}
\end{align}

\emph{Step 4: derivation of the claimed estimate (\ref{eq: lemma - claimed estimate}).}
To conclude, it remains to show that $\sup\limits_{\tau \in [ 0 , \max(t,2\epsilon)]} \; e^{\sigma\tau} \Vert p(\tau) \Vert$ can be replaced by $\sup\limits_{\tau \in [ 0 , t ]} \; e^{\sigma\tau} \Vert p(\tau) \Vert$ in (\ref{eq: proof prel lemma - int 8}). This is obviously true for $t \geq 2\epsilon$, as well as $t = 0$ because $M \geq 1$. Thus, we focus on the case $0 < t < 2\epsilon$. Let $T \in ( 0 , 2\epsilon )$ be arbitrary. Let $\zeta_n \in \mathcal{C}^0(\mathbb{R}_+;\mathbb{R})$ with $0 \leq \zeta_n \leq 1$, $\left. \zeta_n \right\vert_{[0,T]} = 1$ and $\left. \zeta_n \right\vert_{[T+(2\epsilon-T)/n,+\infty)} = 0$ for $n \geq 1$. We define $p_n = \zeta_n p \in \mathcal{C}^0(\mathbb{R}_+;\mathbb{K})$ and we denote by $x_n$ the solution of (\ref{eq: prel lemma}) associated with the initial condition $x_0$ and the disturbance $p_n$. As $p_n(t) = p(t)$ for all $0 \leq t \leq T$, we obtain that $x_n(t) = x(t)$ for all $0 \leq t \leq T$. Therefore, we obtain by applying (\ref{eq: proof prel lemma - int 8}) to $x_n$ at time $t = T$ that, for all $n \geq 1$, 
\begin{align*}
& \sup\limits_{\tau \in [ 0 , T ]} \; e^{\sigma \tau} \Vert x(\tau) \Vert \\
& \leq
M \sup\limits_{\tau \in [ -r-\epsilon , 0 ]} \; \Vert x_0(\tau) \Vert
+ N \sup\limits_{\tau \in [ 0 , 2\epsilon ]} \; e^{\sigma\tau} \Vert p_n(\tau) \Vert \\
& \underset{n \rightarrow + \infty}{\longrightarrow} M \sup\limits_{\tau \in [ -r-\epsilon , 0 ]} \; \Vert x_0(\tau) \Vert
+ N \sup\limits_{\tau \in [ 0 , T ]} \; e^{\sigma\tau} \Vert p(\tau) \Vert ,
\end{align*}
where the limit holds by a continuity argument. Thus, the claimed estimate (\ref{eq: lemma - claimed estimate}) holds.
\end{proof}

\subsection{Study of the truncated model}

We apply the result of Lemma~\ref{lem: prel lemma} to the study of the finite-dimensional truncated model composed of (\ref{eq: ODE satisfies by Y}) and (\ref{def: boundary control system - closed-loop - cont input}).

\begin{lemma}\label{lem: ISS estimate finite dimensional part - preliminary}
Under the assumptions of Theorem~\ref{thm: ISS mild solutions}, there exist $\sigma , C_1 , C_2, C_3, \overline{C}_4, \overline{C}_5, \overline{C}_6 > 0$ such that, for any $D \in \mathcal{C}^1(\mathbb{R}_+;\mathbb{R})$ with $\vert D - D_0 \vert \leq \delta$, $X_0 \in \mathcal{H}$, and  $d_1,d_2 \in \mathcal{C}^1(\mathbb{R}_+;\mathbb{K}^m)$, $Y$ defined by (\ref{eq: definition Y}), where $(X,u) \in \mathcal{C}^0(\mathbb{R}_+;\mathcal{H}) \times \mathcal{C}^1([-D_0-\delta,+\infty);\mathbb{K}^m)$ is the mild solution of the closed-loop system (\ref{def: boundary control system - closed-loop}) associated with $(D,X_0,d_1,d_2)$, satisfies
\begin{align}
\Vert Y(t) \Vert 
& \leq C_1 e^{-\sigma t} \Vert X_0 \Vert_\mathcal{H} 
+ C_2 \sup\limits_{\tau \in [0,t]} e^{- \sigma (t-\tau)} \Vert d_1(\tau) \Vert \label{eq: ISS estimate Y} \\
& \phantom{\leq}\; + C_3 \sup\limits_{\tau \in [0,\max(t-(D_0-\delta),0)]} e^{- \sigma (t-\tau)} \Vert d_2(\tau) \Vert \nonumber
\end{align}
for all $t \geq 0$. Furthermore, the control law $u$ satisfies (\ref{eq: ISS estimate command}) with $\kappa = \sigma$ for all $t \geq 0$.
\end{lemma}

\begin{proof}
Let $\delta \in (0,D_0)$ satisfying the small gain condition (\ref{eq: thm infinite dim system - small gain condition}) be given. Let $\sigma,N > 0$ and $M \geq 1$ be the constants provided by Lemma~\ref{lem: prel lemma} for $A = A_\mathrm{cl}$, $C = B_{N_0} K$, $q = 1$, $r = D_0$, and $\epsilon = \delta$. We introduce the Artstein transformation~\cite{artstein1982linear,richard2003time} by defining, for all $t \geq 0$,  
\begin{equation}\label{eq: Artstein transformation}
Z(t) = Y(t) + \int_{t-D_0}^t e^{(t-D_0-s)A_{N_0}} B_{N_0} u(s) \diff s .
\end{equation}
As $u(\tau) = 0$ for $\tau \leq 0$, we obtain that $u = \varphi K Z + \varphi d_2$. Taking the time derivative, (\ref{eq: ODE satisfies by Y}) yields for all $t \geq 0$
\begin{align}
\dot{Z}(t) 
& = \left\{ A_{N_0} + \varphi(t) e^{-D_0 A_{N_0}} B_{N_0} K \right\} Z(t) \label{eq: lemma finite-dim - ODE Z} \\
& \phantom{=}\; + B_{N_0} K \left\{ [\varphi Z](t-D(t)) - [\varphi Z](t-D_0) \right\} \nonumber \\
& \phantom{=}\; + B_{N_0} d_1(t) + \varphi(t) e^{-D_0 A_{N_0}} B_{N_0} d_2(t) \nonumber \\
& \phantom{=}\; + B_{N_0} \left\{ [\varphi d_2](t-D(t)) - [\varphi d_2](t-D_0) \right\} . \nonumber
\end{align}
We first study the case $t \geq t_1 \triangleq t_0 + D_0 + \delta$. We have for all $t \geq t_1$ that $\varphi(t)=\varphi(t-D_0)=\varphi(t-D(t))=1$ and thus
\begin{align*}
\dot{Z}(t) 
& = A_\mathrm{cl} Z(t) + B_{N_0} K \left\{ Z(t-D(t)) - Z(t-D_0) \right\} \\
& \phantom{=}\; + B_{N_0} d_1(t) + e^{-D_0 A_{N_0}} B_{N_0} d_2(t) \\
& \phantom{=}\; + B_{N_0} \left\{ d_2(t-D(t)) - d_2(t-D_0) \right\} .
\end{align*}
Consequently, by applying Lem.~\ref{lem: prel lemma} to the above ODE with $x(t) = Z(t+t_1)$, it follows from (\ref{eq: lemma - claimed estimate}) that, for all $t \geq t_1$,
\begin{align}
& \Vert Z(t) \Vert \label{eq: estimatation finite dimensional part Z - t geq t0} \\
& \leq M e^{- \sigma (t-t_1)} \sup\limits_{\tau \in [ t_0 , t_1 ]} \Vert Z(\tau) \Vert
+ \tilde{N}_1 \sup\limits_{\tau \in [t_1,t]} e^{-\sigma (t-\tau)} \Vert d_1(\tau) \Vert \nonumber \\
& \phantom{\leq}\, + \tilde{N}_2 \sup\limits_{\tau \in [t_0,t]} e^{-\sigma (t-\tau)} \Vert d_2(\tau) \Vert , \nonumber
\end{align}
with $\tilde{N}_1 = N \Vert B_{N_0} \Vert$ and 
\begin{equation*}
\tilde{N}_2 = N \left\{ \Vert e^{-D_0 A_{N_0}} B_{N_0} \Vert + \Vert B_{N_0} \Vert e^{\sigma D_0} ( e^{\sigma\delta} + 1 ) \right\} .
\end{equation*}
In the case $0 \leq t \leq t_1$, based on (\ref{eq: lemma finite-dim - ODE Z}), a standard application of Gr{\"o}nwall's inequality shows the existence of $\gamma_0,\gamma_1,\gamma_2 > 0$, which only depend of $A_{N_0},B_{N_0},K,D_0,\delta,t_0,\sigma$, such that
\begin{align}
\Vert Z(t) \Vert 
& \leq \gamma_0 e^{-\sigma t} \Vert Y(0) \Vert 
+ \gamma_1 \sup\limits_{\tau \in [0,t]} e^{-\sigma (t-\tau)} \Vert d_1(\tau) \Vert \label{eq: estimatation finite dimensional part Z - t leq t0 - bis} \\
& \phantom{\leq}\; + \gamma_2 \sup\limits_{\tau \in [0,t]} e^{-\sigma (t-\tau)} \Vert d_2(\tau) \Vert , \nonumber
\end{align}
for all $0 \leq t \leq t_1$, where it has been used $Z(0)=Y(0)$. Thus, combining (\ref{eq: estimatation finite dimensional part Z - t geq t0}-\ref{eq: estimatation finite dimensional part Z - t leq t0 - bis}), we obtain that, for all $t \geq 0$, 
\begin{align}
\Vert Z(t) \Vert 
& \leq \gamma_3 e^{-\sigma t} \Vert Y(0) \Vert 
+ \gamma_4 \sup\limits_{\tau \in [0,t]} e^{-\sigma (t-\tau)} \Vert d_1(\tau) \Vert \label{eq: ISS estimate Z} \\
& \phantom{\leq}\; + \gamma_5 \sup\limits_{\tau \in [0,t]} e^{-\sigma (t-\tau)} \Vert d_2(\tau) \Vert \nonumber 
\end{align}
with $\gamma_3 = M e^{\sigma t_1} \gamma_0$, $\gamma_4 = M e^{\sigma t_1} \gamma_1 + \tilde{N}_1$, and $\gamma_5 = M e^{\sigma t_1} \gamma_2 + \tilde{N}_2$. From $u = \varphi K Z + \varphi d_2$ and (\ref{eq: Riesz basis - inequality}), we deduce that (\ref{eq: ISS estimate command}) holds for $\kappa = \sigma$ with $\overline{C}_4 = \Vert K \Vert \gamma_3 / \sqrt{m_R}$, $\overline{C}_5 = \Vert K \Vert \gamma_4$, and $\overline{C}_6 = \Vert K \Vert \gamma_5 + 1$. Now, using estimates (\ref{eq: Riesz basis - inequality}), (\ref{eq: ISS estimate command}) with $\kappa = \sigma$, and (\ref{eq: ISS estimate Z}) into (\ref{eq: Artstein transformation}) and the fact that, for $i \in \{1,2\}$,
\begin{align*}
& \int_{\max(t-D_0,0)}^t \sup\limits_{\tau \in [0,s]} e^{-\sigma (s-\tau)} \Vert d_i(\tau) \Vert \diff s \\
& \qquad\qquad\qquad\leq D_0 e^{\sigma D_0} \sup\limits_{\tau \in [0,t]} e^{- \sigma (t-\tau)} \Vert d_i(\tau) \Vert ,
\end{align*}
we obtain that
\begin{align}
\Vert Y(t) \Vert 
& \leq C_1 e^{-\sigma t} \Vert X_0 \Vert_\mathcal{H} 
+ C_2 \sup\limits_{\tau \in [0,t]} e^{- \sigma (t-\tau)} \Vert d_1(\tau) \Vert \label{eq: prel ISS estimate Y} \\
& \phantom{\leq}\; + C_3 \sup\limits_{\tau \in [0,t]} e^{- \sigma (t-\tau)} \Vert d_2(\tau) \Vert \nonumber 
\end{align}
for all $t \geq 0$, with 
\begin{align*}
C_1 & = \gamma_3 / \sqrt{m_R} + e^{D_0 (\sigma + \Vert A_{N_0} \Vert)} \Vert B_{N_0} \Vert \overline{C}_4 / \sigma , \\
C_2 & = \gamma_4 + D_0 e^{D_0 ( \sigma + \Vert A_{N_0} \Vert )} \Vert B_{N_0} \Vert \overline{C}_5 , \\
C_3 & = \gamma_5 + D_0 e^{D_0 ( \sigma + \Vert A_{N_0} \Vert )} \Vert B_{N_0} \Vert \overline{C}_6 .
\end{align*}
To conclude the proof, it remains to show that we can substitute in estimate (\ref{eq: prel ISS estimate Y}) the term $\sup\limits_{\tau \in [0,t]} e^{- \sigma (t-\tau)} \Vert d_2(\tau) \Vert$ by the term $\sup\limits_{\tau \in [0,\max(t-(D_0-\delta),0)]} e^{- \sigma (t-\tau)} \Vert d_2(\tau) \Vert$. To do so, let $T \geq 0$ be arbitrary. Let $\zeta_n \in \mathcal{C}^1(\mathbb{R};\mathbb{R})$ with $0 \leq \zeta_n \leq 1$ be such that $\left. \zeta_n \right\vert_{(-\infty,T-(D_0-\delta)]} = 1$ and $\left. \zeta_n \right\vert_{[T-(D_0-\delta)+(D_0-\delta)/n,+\infty]} = 0$ for $n \geq 1$. Then we define $d_{2,n} = \zeta_n d_2  \in \mathcal{C}^1(\mathbb{R}_+;\mathbb{K}^m)$. Thus we can introduce $(X_n,u_n)$ the mild solution of (\ref{def: boundary control system - closed-loop}) associated with $(D,X_0,d_1,d_{2,n})$. From~\cite[Sec.~III.C]{lhachemi2019feedback} (see also~\cite{bresch2018new} in the case $\varphi=1$), it can be seen that $\left. X \right\vert_{[0,T]} = \left. X_n \right\vert_{[0,T]}$ and $\left. u \right\vert_{[-D_0-\delta,T-(D_0-\delta)]} = \left. u_n \right\vert_{[-D_0-\delta,T-(D_0-\delta)]}$. Applying the ISS estimate (\ref{eq: prel ISS estimate Y}) at time $t = T$ to $X_n$ for any $n \geq 1$, we obtain that 
\begin{align*}
& \Vert Y(T) \Vert \\
& \leq C_1 e^{-\sigma T} \Vert X_0 \Vert_\mathcal{H} 
+ C_2 \sup\limits_{\tau \in [0,T]} e^{- \sigma (T-\tau)} \Vert d_1(\tau) \Vert \\
& \phantom{\leq}\; + C_3 \sup\limits_{\tau \in [0,T]} e^{- \sigma (T-\tau)} \Vert \zeta_n(\tau) d_2(\tau) \Vert .
\end{align*}
By letting $n \rightarrow + \infty$, a continuity argument shows that (\ref{eq: ISS estimate Y}) holds at time $t = T$. As $T \geq 0$ is arbitrary, this concludes the proof.
\end{proof}

\section{Exponential ISS of the infinite-dimensional system}\label{sec: exponential ISS of the infinite-dimensional system}

This section is devoted to the proof of the main result of this paper: namely, Theorem~\ref{thm: ISS mild solutions}. Let $\sigma > 0$ be provided by Lemma~\ref{lem: ISS estimate finite dimensional part - preliminary}. Let $0 < \kappa < \min(\alpha,\sigma)$ be given and define $\epsilon = \kappa / \alpha \in (0,1)$. First, we infer from $\xi = \sup\limits_{n \geq N_0 +1} \left\vert \dfrac{\lambda_n}{\operatorname{Re} \lambda_n} \right\vert < \infty$ that the following estimate holds true for all $n \geq N_0 + 1$
\begin{align}
\left\vert \dfrac{\lambda_n}{\operatorname{Re} \lambda_n + \kappa} \right\vert
& = \left\vert \dfrac{\lambda_n}{\operatorname{Re} \lambda_n} \right\vert \times \left\vert \dfrac{\operatorname{Re} \lambda_n}{\operatorname{Re} \lambda_n + \kappa} \right\vert \nonumber \\
& \leq \xi \left\vert \dfrac{- \alpha}{-\alpha + \kappa} \right\vert
\leq \dfrac{\alpha \xi}{\alpha - \kappa} , \label{eq: estimate abs(lambda_n/(Re(lambda_n)+kappa) for 1 leq k leq N0}
\end{align}
where we have used the facts that $\operatorname{Re} \lambda_n \leq - \alpha < - \kappa < 0$ and the function $x \rightarrow x/(x+\kappa)$ is positive and strictly increasing for $x \in (-\infty , - \kappa)$. Now, from the integral expression (\ref{eq: coeff in Riesz basis integral form}) of the coefficient of projection $c_n$ and 
using the definition of the boundary input (\ref{eq: boundary input}), we have for all $t \geq 0$
\begin{align}
& \vert c_n(t) \vert \nonumber \\
& \leq 
e^{\operatorname{Re} \lambda_n t} \vert c_n(0) \vert \label{eq: intermediate estimate 1}\\
& \phantom{\leq}\; + \vert \lambda_n \vert \int_0^t e^{\operatorname{Re} \lambda_n (t-\tau)} \vert \left< B u(\tau-D(\tau)) , \psi_n \right>_\mathcal{H} \vert \diff\tau \nonumber \\
& \phantom{\leq}\; + \int_0^t e^{\operatorname{Re} \lambda_n (t-\tau)} \vert \left< \mathcal{A} B u(\tau-D(\tau)) , \psi_n \right>_\mathcal{H} \vert \diff\tau \nonumber \\
& \phantom{\leq}\; + \vert \lambda_n \vert \int_0^t e^{\operatorname{Re} \lambda_n (t-\tau)} \vert \left< B d_1(\tau) , \psi_n \right>_\mathcal{H} \vert \diff\tau \nonumber \\ 
& \phantom{\leq}\; + \int_0^t e^{\operatorname{Re} \lambda_n (t-\tau)} \vert \left< \mathcal{A} B d_1(\tau) , \psi_n \right>_\mathcal{H} \vert \diff\tau . \nonumber
\end{align}
We now evaluate the different terms on the right hand side of (\ref{eq: intermediate estimate 1}). Denoting by $e_1, \ldots , e_m$ the canonical basis of $\mathbb{K}^m$, we have $d_1(t) = \sum\limits_{k=1}^{m} d_{1,k}(t) e_k$ with $d_{1,k}(t) \in \mathbb{K}$. Then, noting that $\vert d_{1,k}(t) \vert \leq \Vert d_1(t) \Vert$, we have for all $n \geq N_0 + 1$ and $t \geq 0$
\begin{align}
& \vert \lambda_n \vert \int_0^t e^{\operatorname{Re} \lambda_n (t-\tau)} \vert \left< B d_1(\tau) , \psi_n \right>_\mathcal{H} \vert \diff\tau \nonumber \\
& \leq 
\vert \lambda_n \vert \sum\limits_{k=1}^{m} \vert \left< B e_k , \psi_n \right>_\mathcal{H} \vert \int_0^t e^{\operatorname{Re} \lambda_n (t-\tau)} \Vert d_{1}(\tau) \Vert \diff\tau \nonumber \\
& \leq
\vert \lambda_n \vert \sum\limits_{k=1}^{m} \vert \left< B e_k , \psi_n \right>_\mathcal{H} \vert \int_0^t e^{(1-\epsilon) \operatorname{Re} \lambda_n (t-\tau)} \diff\tau \nonumber \\
& \phantom{\leq}\; \times \sup\limits_{\tau \in [0,t]} e^{\epsilon \operatorname{Re} \lambda_n (t-\tau)} \Vert d_{1}(\tau) \Vert \nonumber \\
& \leq
\dfrac{1}{1-\epsilon} \left\vert \dfrac{\lambda_n}{\operatorname{Re}\lambda_n} \right\vert \sum\limits_{k=1}^{m} \vert \left< B e_k , \psi_n \right>_\mathcal{H} \vert \sup\limits_{\tau \in [0,t]} e^{-\epsilon\alpha (t-\tau)} \Vert d_{1}(\tau) \Vert \nonumber \\
& \leq
\dfrac{\alpha\xi}{\alpha-\kappa} \sum\limits_{k=1}^{m} \vert \left< B e_k , \psi_n \right>_\mathcal{H} \vert \sup\limits_{\tau \in [0,t]} e^{-\kappa (t-\tau)} \Vert d_{1}(\tau) \Vert . \label{eq: intermediate estimate - d1 - 1st integral term}
\end{align}
Similarly, we have for all $n \geq N_0 + 1$ and $t \geq 0$
\begin{align}
& \int_0^t e^{\operatorname{Re} \lambda_n (t-\tau)} \vert \left< \mathcal{A} B d_1(\tau) , \psi_n \right>_\mathcal{H} \vert \diff\tau \nonumber \\
& \leq \dfrac{1}{\alpha-\kappa} \sum\limits_{k=1}^{m} \vert \left< \mathcal{A} B e_k , \psi_n \right>_\mathcal{H} \vert \sup\limits_{\tau \in [0,t]} e^{-\kappa (t-\tau)} \Vert d_1(\tau) \Vert \label{eq: intermediate estimate - d1 - 2nd integral term} .
\end{align}
We now estimate the two remaining integral terms involving the control input $u$ on the right-hand side of (\ref{eq: intermediate estimate 1}).  We note that these two integrals are null for $t \leq D_0 - \delta$ because $u(\tau)=0$ for $\tau \leq 0$. Thus we focus on the case $t \geq D_0-\delta$. First, we evaluate the following integral:
\begin{align*}
\mathcal{I}_n(t) 
& = \int_{0}^t e^{-\operatorname{Re} \lambda_n \tau} \Vert u(\tau - D(\tau)) \Vert \diff\tau \\ 
& = \int_{0}^t e^{-\operatorname{Re} \lambda_n \tau} \chi_{\{s-D(s) \geq 0 \}}(\tau) \Vert u(\tau - D(\tau)) \Vert \diff\tau
\end{align*}
for $t \geq D_0-\delta$. To do so, we note that
\begin{align*}
& \int_{0}^t e^{-\operatorname{Re} \lambda_n \tau}  \chi_{\{s-D(s) \geq 0 \}}(\tau) \sup\limits_{s \in [0,\tau-D(\tau)]} e^{-\kappa ((\tau-D(\tau))-s)} \Vert d_i(s) \Vert \diff\tau \\
& \leq \dfrac{e^{\kappa (D_0+\delta)}}{\vert \operatorname{Re} \lambda_n+\kappa \vert} e^{- ( \operatorname{Re} \lambda_n+\kappa) t} \sup\limits_{s \in [0,t-(D_0-\delta)]} e^{\kappa s} \Vert d_i(s) \Vert .
\end{align*}
Therefore, recalling that $0 < \kappa < \sigma$, the use of (\ref{eq: ISS estimate command}) provided by Lemma~\ref{lem: ISS estimate finite dimensional part - preliminary} yields
\begin{equation*}
\mathcal{I}_n(t) \leq \dfrac{e^{\kappa (D_0+\delta)}}{\vert \operatorname{Re} \lambda_n + \kappa \vert} e^{-\operatorname{Re} \lambda_n t} \Delta(t) 
\end{equation*}
with
\begin{align*}
\Delta(t) & = 
\overline{C}_4 e^{-\kappa t} \Vert X_0 \Vert_\mathcal{H} + \overline{C}_5 \sup\limits_{s \in [0,t-(D_0-\delta)]} e^{-\kappa (t-s)} \Vert d_1(s) \Vert \\
& \phantom{=}\; + \overline{C}_6 \sup\limits_{s \in [0,t-(D_0-\delta)]} e^{-\kappa (t-s)} \Vert d_2(s) \Vert. 
\end{align*}
With $u(t) = \sum\limits_{k=1}^{m} u_k(t) e_k$  where $u_k(t) \in \mathbb{K}$, we have that $\vert u_k(t) \vert \leq \Vert u(t) \Vert$ for all $1 \leq k \leq m$. Recalling that $\operatorname{Re} \lambda_n \leq - \alpha < - \kappa < 0$ for any $n \geq N_0 +1$, we obtain that, for all $n \geq N_0 +1$ and $t \geq D_0-\delta$,
\begin{align}
& \vert \lambda_n \vert \int_0^t e^{\operatorname{Re} \lambda_n (t-\tau)} \vert \left< B u(\tau-D(\tau)) , \psi_n \right>_\mathcal{H} \vert \diff\tau \nonumber \\
& \leq \vert \lambda_n \vert \sum\limits_{k=1}^{m} \vert \left< B e_k , \psi_n \right>_\mathcal{H} \vert e^{\operatorname{Re} \lambda_n t} \mathcal{I}_n(t) \nonumber \\
& \leq \dfrac{\alpha \xi e^{\kappa (D_0+\delta)}}{\alpha-\kappa} \sum\limits_{k=1}^{m} \vert \left< B e_k , \psi_n \right>_\mathcal{H} \vert \Delta(t) \label{eq: intermediate estimate - Z - 1st integral term} 
\end{align}
and
\begin{align}
& \int_0^t e^{\operatorname{Re} \lambda_n (t-\tau)} \vert \left< \mathcal{A} B u(\tau-D(\tau)) , \psi_n \right>_\mathcal{H} \vert \diff\tau \nonumber \\
& \leq \sum\limits_{k=1}^{m} \vert \left< \mathcal{A} B e_k , \psi_n \right>_\mathcal{H} \vert e^{\operatorname{Re} \lambda_n t} \mathcal{I}_n(t) \nonumber \\
& \leq \dfrac{e^{\kappa (D_0+\delta)}}{\alpha-\kappa} \sum\limits_{k=1}^{m} \vert \left< \mathcal{A} B e_k , \psi_n \right>_\mathcal{H} \vert \Delta(t) . \label{eq: intermediate estimate - Z - 2nd integral term}
\end{align}
Based on (\ref{eq: intermediate estimate - d1 - 1st integral term}-\ref{eq: intermediate estimate - Z - 2nd integral term}), we deduce from (\ref{eq: intermediate estimate 1}), Young's inequality, estimates (\ref{eq: Riesz basis - inequality}), and the fact $\Vert Y(0) \Vert^2 
= \sum\limits_{n=1}^{N_0} \vert c_n(0) \vert^2$ that, for all $t \geq 0$, 
\begin{align}
\sum\limits_{n \geq N_0 + 1} \vert c_n(t) \vert^2 
& \leq \tilde{C}_1 e^{- 2 \kappa t} \Vert X_0 \Vert_\mathcal{H}^2 
+ \tilde{C}_2 \sup\limits_{\tau \in [0,t]} e^{-2\kappa(t-\tau)} \Vert d_1(\tau) \Vert^2 \nonumber \\
& \phantom{\leq}\; + \tilde{C}_3 \sup\limits_{\tau \in [0,\max(t-(D_0-\delta),0)]} e^{-2\kappa(t-\tau)} \Vert d_2(\tau) \Vert^2 , \label{eq: ISS estimate uncontrolled infinite-dimensional part}
\end{align}
where constants $\tilde{C}_i$ are given by 
\begin{subequations}
\begin{align*}
\tilde{C}_0 & = \alpha^2 \xi^2 \sum\limits_{k=1}^{m} \Vert B e_k \Vert_\mathcal{H}^2 + \sum\limits_{k=1}^{m} \Vert \mathcal{A} B e_k \Vert_\mathcal{H}^2 , \\
\tilde{C}_1 & = \dfrac{4}{m_R} \left( 1 + \dfrac{2 m \overline{C}_4^2 e^{2 \kappa (D_0+\delta)}}{(\alpha-\kappa)^2} \tilde{C}_0 \right) , \\
\tilde{C}_2 & = \dfrac{8m\left(1+\overline{C}_5 e^{ \kappa (D_0+\delta)}\right)^2}{m_R (\alpha-\kappa)^2} \tilde{C}_0 , \\
\tilde{C}_3 & = \dfrac{8m \overline{C}_6^2 e^{2\kappa(D_0+\delta)}}{m_R (\alpha-\kappa)^2} \tilde{C}_0 .
\end{align*}
\end{subequations}
Consequently, as
\begin{align*}
\Vert X(t) \Vert_\mathcal{H}
& \leq \sqrt{M_R \sum\limits_{n \geq 1} \vert c_n(t) \vert^2} \\
& \leq \sqrt{M_R} \left( \Vert Y(t) \Vert + \sqrt{\sum\limits_{n \geq N_0 + 1} \vert c_n(t) \vert^2} \right) ,
\end{align*}
we obtain from (\ref{eq: ISS estimate Y}) and (\ref{eq: ISS estimate uncontrolled infinite-dimensional part}) that the ISS estimate (\ref{eq: ISS estimate}) holds with $\overline{C}_i = \sqrt{M_R} \left( C_i + \sqrt{\tilde{C}_i} \right)$, $1 \leq i \leq 3$, which concludes the proof of Theorem~\ref{thm: ISS mild solutions}.

\section{Extension of the main result to continuous boundary perturbations}\label{sec: exponential ISS for relaxed initial conditions and perturbations}

The result stated in Theorem~\ref{thm: ISS mild solutions} deals with mild solutions associated with continuously differentiable boundary disturbances. However, as shown in~\cite{lhachemi2018iss}, the satisfaction of an ISS estimate, combined with the introduction of a proper concept of weak solution, can be employed to easily extend the obtained ISS estimate to boundary disturbances exhibiting relaxed regularity assumptions. Such a concept of weak solutions extends to abstract boundary control systems the concept
of weak solutions originally introduced for infinite-dimensional
nonhomogeneous Cauchy problems in~\cite{ball1977strongly} and further investigated in~\cite[Def.~3.1.6, Thm.~3.1.7, A.5.29]{Curtain2012} under a variational from. In this context and adopting the approach reported in~\cite{lhachemi2018iss}, we introduce the following concept of weak solution for the closed-loop dynamics (\ref{def: boundary control system - closed-loop}). 

\begin{definition}\label{def: weak solution}
Let $(\mathcal{A},\mathcal{B})$ be an abstract boundary control system such that Assumption~\ref{assum: A1} holds. Let $t_0,D_0 > 0$, $\delta \in (0,D_0)$, a transition signal $\varphi \in \mathcal{C}^1(\mathbb{R};\mathbb{R})$ over $[0,t_0]$, and $K \in \mathbb{K}^{m \times N_0}$ be arbitrary. For a time-varying delay $D \in \mathcal{C}^1(\mathbb{R}_+;\mathbb{R})$ with $\vert D - D_0 \vert \leq \delta$, an initial condition $X_0 \in \mathcal{H}$, and boundary perturbations $d_1,d_2\in\mathcal{C}^0(\mathbb{R}_+;\mathbb{K}^m)$, we say that $(X,u) \in \mathcal{C}^0(\mathbb{R}_+;\mathcal{H}) \times \mathcal{C}^0(\mathbb{R}_+;\mathbb{K}^m)$ is a weak solution of (\ref{def: boundary control system - closed-loop}) associated with $(D,X_0,d_1,d_2)$ if for any $T > 0$ and any $z \in \mathcal{C}^0([0,T];D(\mathcal{A}^*_0)) \cap \mathcal{C}^1([0,T];\mathcal{H})$ with\footnote{Such a function $z$ is called a test function over $[0,T]$.} $\mathcal{A}_0^* z \in \mathcal{C}^0([0,T];\mathcal{H})$ and $z(T)=0$, we have: 
\begin{align}
& \int_0^T \left< X(t), \mathcal{A}_0^* z(t) + \dfrac{\mathrm{d} z}{\mathrm{d} t}(t) \right>_\mathcal{H} \diff t \nonumber \\
& = - \left< X_0 , z(0) \right>_\mathcal{H} \label{eq: def weak solution} \\
& \phantom{=}\; + \int_0^T \left< B ( u(t-D(t)) + d_1(t) ) , \mathcal{A}_0^* z(t) \right>_\mathcal{H} \diff t \nonumber \\
& \phantom{=}\; - \int_0^T \left< \mathcal{A} B ( u(t-D(t)) + d_1(t) ) , z(t) \right>_\mathcal{H} \diff t \nonumber ,
\end{align}
with $B$ an arbitrarily given lifting operator associated with $(\mathcal{A},\mathcal{B})$ and where the control input $u$ satisfies $\left. u \right\vert_{[-D_0-\delta,0]} = 0$ and, for all $t \geq 0$, 
\begin{align}
u(t) & = \varphi(t) \Bigg\{  K Y(t) + d_2(t) \label{eq: def weak solution - control input} \\
& \phantom{=}\; \hspace{1cm} + K \int_{\max(t-D_0,0)}^{t} e^{(t-s-D_0)A_{N_0}} B_{N_0} u(s) \diff s  \Bigg\} \nonumber
\end{align}
with $Y$ defined by (\ref{eq: definition Y}).
\end{definition}

In particular, using the definition of the mild solutions (\ref{eq: def mild solution}) into the left hand side of (\ref{eq: def weak solution}), it is easy to show (see Appendix~\ref{annex: mild sol are weak sol} for details) that any mild solution is also a weak solution.

\begin{remark}
Following~\cite[Sec.~4]{lhachemi2018iss}, we have the following facts. 
\begin{itemize}
\item Definition~\ref{def: weak solution} is independent of a specifically selected lifting operator in the sense that the right hand side of (\ref{eq: def weak solution}) is unchanged when switching between different lifting operators $B$ associated with $(\mathcal{A},\mathcal{B})$.
\item $X_0$ is the initial condition of the weak solution in the sense that if $(X,u) \in \mathcal{C}^0(\mathbb{R}_+;\mathcal{H}) \times \mathcal{C}^0(\mathbb{R}_+;\mathbb{K}^m)$ is a weak solution of (\ref{def: boundary control system - closed-loop}) associated with $(D,X_0,d_1,d_2)$, then we have $X(0) = X_0$.
\end{itemize}
\end{remark}

We first state a preliminary result about the uniqueness of the weak solutions for the studied problem.

\begin{lemma}\label{lem: uniqueness weak solutions}
For any $D \in \mathcal{C}^1(\mathbb{R}_+;\mathbb{R})$ with $\vert D - D_0 \vert \leq \delta$, $X_0 \in \mathcal{H}$, and $d_1,d_2 \in \mathcal{C}^0(\mathbb{R}_+;\mathbb{K}^m)$, there exists at most one weak solution $(X,u) \in \mathcal{C}^0(\mathbb{R}_+;\mathcal{H}) \times \mathcal{C}^0(\mathbb{R}_+;\mathbb{K}^m)$ of the closed-loop system (\ref{def: boundary control system - closed-loop}) associated with $(D,X_0,d_1,d_2)$.
\end{lemma}

\begin{proof}
By linearity, it is sufficient to show that if $(X,u) \in \mathcal{C}^0(\mathbb{R}_+;\mathcal{H}) \times \mathcal{C}^0(\mathbb{R}_+;\mathbb{K}^m)$ satisfies 
\begin{align*}
& \int_0^T \left< X(t), \mathcal{A}_0^* z(t) + \dfrac{\mathrm{d} z}{\mathrm{d} t}(t) \right>_\mathcal{H} \diff t \\
& = 
\int_0^T \left< B u(t-D(t)) , \mathcal{A}_0^* z(t) \right>_\mathcal{H} \diff t \\
& \phantom{=}\; - \int_0^T \left< \mathcal{A} B u(t-D(t)) , z(t) \right>_\mathcal{H} \diff t \nonumber 
\end{align*}
for all $T>0$ and for all test function $z$ over $[0,T]$ with $\left. u \right\vert_{[-D_0-\delta,0]} = 0$ and 
\begin{equation*}
u(t) = \varphi(t) K \left\{  Y(t) + \int_{\max(t-D_0,0)}^{t} e^{(t-s-D_0)A_{N_0}} B_{N_0} u(s) \diff s \right\}
\end{equation*}
for all $t \geq 0$, then $X=0$ and $u=0$. We proceed by induction by showing that $\left. X \right\vert_{[0,n(D_0-\delta)] = 0}$ and $\left. u \right\vert_{[-D_0-\delta,(n-1)(D_0-\delta)] = 0}$ for all $n \geq 1$.

\textit{Initialization:} From $\left. u \right\vert_{[-D_0-\delta,0]} = 0$, we obtain for $T = D_0 - \delta$ that: 
\begin{equation*}
\int_0^{D_0-\delta} \left< X(t), \mathcal{A}_0^* z(t) + \dfrac{\mathrm{d} z}{\mathrm{d} t}(t) \right>_\mathcal{H} \diff t 
= 0
\end{equation*}
for any test function $z$ over $[0,D_0-\delta]$ because $t - D(t) \leq 0$ and thus $u(t-D(t))=0$ for all $t \leq D_0-\delta$. Using the test function $z(t) = e^{-\overline{\lambda}_k t} \int_{D_0 - \delta}^t \left< X(\tau) , \psi_k \right>_\mathcal{H} e^{\overline{\lambda}_k\tau} \diff\tau \, \psi_k$, we obtain that $\mathcal{A}_0^* z(t) + \dfrac{\mathrm{d} z}{\mathrm{d} t}(t) = \left< X(t) , \psi_k \right>_\mathcal{H} \psi_k$ and thus $\int_0^{D_0-\delta} \vert \left< X(t) , \psi_k \right>_\mathcal{H} \vert^2 \diff t = 0$. By continuity of $X$, we infer that $\left< X(t) , \psi_k \right>_\mathcal{H} = 0$ for all $t \in [0,D_0-\delta]$ and all $k \geq 1$. Then (\ref{eq: series expansion Riesz basis}) yields $X(t) = 0$ for all $t \in [0,D_0-\delta]$.

\textit{Heredity:} Let $n \geq 1$ be such that $\left. X \right\vert_{[0,n(D_0-\delta)]} = 0$ and $\left. u \right\vert_{[-D_0-\delta,(n-1)(D_0-\delta)] = 0}$. Then we have $\left. Y \right\vert_{[0,n(D_0-\delta)]} = 0$ and thus
\begin{equation*}
u(t) = \varphi(t) K \int_{\max(t-D_0,0)}^{t} e^{(t-s-D_0)A_{N_0}} B_{N_0} u(s) \diff s
\end{equation*}
for all $t \in [0,n(D_0-\delta)]$. Hence, using Gr{\"o}nwall's inequality (see also~\cite{bresch2018new} and~\cite[Sec.~III.C]{lhachemi2019feedback}), $u(t) = 0$ for all $t \in [0,n(D_0-\delta)]$. Thus, $u(t-D(t)) = 0$ for all $t \in [0,(n+1)(D_0-\delta)]$ and we obtain with $T = (n+1)(D_0-\delta)$ that
\begin{equation*}
\int_0^{(n+1)(D_0-\delta)} \left< X(t), \mathcal{A}_0^* z(t) + \dfrac{\mathrm{d} z}{\mathrm{d} t}(t) \right>_\mathcal{H} \diff t 
= 0
\end{equation*}
for all test function $z$ over $[0,(n+1)(D_0-\delta)]$. Using the test function $z(t) = e^{-\overline{\lambda}_k t} \int_{(n+1)(D_0-\delta)}^t \left< X(\tau) , \psi_k \right>_\mathcal{H} e^{\overline{\lambda}_k \tau} \diff\tau \, \psi_k$, we obtain that $\mathcal{A}_0^* z(t) + \dfrac{\mathrm{d} z}{\mathrm{d} t}(t) = \left< X(t) , \psi_k \right>_\mathcal{H} \psi_k$ and thus we infer that $\int_0^{(n+1)(D_0-\delta)} \vert \left< X(t) , \psi_k \right>_\mathcal{H} \vert^2 \diff t = 0$. By continuity of $X$, we infer that $\left< X(t) , \psi_k \right>_\mathcal{H} = 0$ for all $t \in [0,(n+1)(D_0-\delta)]$ and all $k \geq 1$. We deduce from (\ref{eq: series expansion Riesz basis}) that $X(t) = 0$ for all $t \in [0,(n+1)(D_0-\delta)]$. This completes the proof by induction.
\end{proof}

We can now state the main result of this section whose proof is an adapation of~\cite[][Thm.~3]{lhachemi2018iss}.

\begin{theorem}
In the context of both assumptions and conclusions of Theorem~\ref{thm: ISS mild solutions}, for any $D \in \mathcal{C}^1(\mathbb{R}_+;\mathbb{R})$ with $\vert D - D_0 \vert \leq \delta$, $X_0 \in \mathcal{H}$, and  $d_1,d_2 \in \mathcal{C}^0(\mathbb{R}_+;\mathbb{K}^m)$, there exists a unique weak solution $(X,u) \in \mathcal{C}^0(\mathbb{R}_+;\mathcal{H}) \times \mathcal{C}^0(\mathbb{R}_+;\mathbb{K}^m)$ associated with $(D,X_0,d_1,d_2)$ of the closed-loop system (\ref{def: boundary control system - closed-loop}). Furthermore, this weak solution satisfies the ISS estimates (\ref{eq: ISS estimate}-\ref{eq: ISS estimate command}) for all $t \geq 0$.
\end{theorem}

\begin{proof}
We consider $\delta \in (0,D_0)$, $\kappa \in (0,\alpha)$, and the constants $\overline{C}_1, \overline{C}_2, \overline{C}_3, \overline{C}_4, \overline{C}_5, \overline{C}_6 > 0$ as provided by Theorem~\ref{thm: ISS mild solutions}. Let $D \in \mathcal{C}^1(\mathbb{R}_+;\mathbb{R})$ with $\vert D - D_0 \vert \leq \delta$, $X_0 \in \mathcal{H}$, and  $d_1,d_2 \in \mathcal{C}^0(\mathbb{R}_+;\mathbb{K}^m)$ be given. The uniqueness follows from Lemma~\ref{lem: uniqueness weak solutions}. We prove the existence.

For a given $T > 0$, as $\mathcal{C}^1([0,T];\mathbb{K}^m)$ is a dense subset of $\mathcal{C}^0([0,T];\mathbb{K}^m)$, we introduce $d_{1,n},d_{2,n} \in \mathcal{C}^1([0,T];\mathbb{K}^m)$ such that $(d_{1,n})_n$ and $(d_{2,n})_n$ converge uniformly over $[0,T]$ to $\left. d_1 \right\vert_{[0,T]}$ and $\left. d_2 \right\vert_{[0,T]}$, respectively. We introduce $(X_n,u_n) \in \mathcal{C}^0([0,T];\mathcal{H}) \times \mathcal{C}^1([-D_0-\delta,T];\mathbb{K}^m)$ the unique mild solution of the closed-loop system (\ref{def: boundary control system - closed-loop}) over $[0,T]$ associated with $(D,X_0,d_{1,n},d_{2,n})$. By linearity of (\ref{def: boundary control system - closed-loop}), $(X_n-X_m,u_n-u_m)$ is the unique mild solution of the closed-loop system (\ref{def: boundary control system - closed-loop}) over $[0,T]$ associated with $(D,0,d_{1,n}-d_{1,m},d_{2,n}-d_{2,m})$. Thus, we deduce from (\ref{eq: ISS estimate}-\ref{eq: ISS estimate command}) that both $(X_n)_n$ and $(u_n)_n$ are Cauchy sequences of the Banach spaces $C^0([0,T];\mathcal{H})$ and $C^0([0,T];\mathbb{K}^m)$, respectively. Thus
$X_{n} \underset{n \rightarrow +\infty}{\longrightarrow} X \in C^0([0,T];\mathcal{H})$ and $u_{n} \underset{n \rightarrow +\infty}{\longrightarrow} u \in C^0([0,T];\mathbb{K}^m)$ and $(X,u)$ satisfies the estimates (\ref{eq: ISS estimate}-\ref{eq: ISS estimate command}) for any $t \in [0,T]$. It is easy to see from (\ref{eq: ISS estimate}-\ref{eq: ISS estimate command}) that the obtained $X$ and $u$ are independent of the selected approximating sequences $(d_{1,n})_n$ and $(d_{2,n})_n$ but only depend on $D$, $X_0$, $d_1$, and $d_2$. 

For any given $0 < T_1 < T_2$, let $(X_1,u_1) \in C^0([0,T_1];\mathcal{H}) \times C^0([0,T_1];\mathbb{K}^m)$ and $(X_2,u_2) \in C^0([0,T_2];\mathcal{H}) \times C^0([0,T_2];\mathbb{K}^m)$ be the result of the above construction over the time intervals $[0,T_1]$ and $[0,T_2]$, respectively. By restricting the approximating sequences of $\left. d_1 \right\vert_{[0,T_2]}$ and $\left. d_2 \right\vert_{[0,T_2]}$ to the interval $[0,T_1]$, we obtain approximating sequences of $\left. d_1 \right\vert_{[0,T_1]}$ and $\left. d_2 \right\vert_{[0,T_1]}$. Then, we infer that $\left. X_2 \right\vert_{[0,T_1]} = X_1$ and $\left. u_2 \right\vert_{[0,T_1]} = u_1$. Consequently, we can define $(X,u) \in \mathcal{C}^0(\mathbb{R}_+;\mathcal{H}) \times \mathcal{C}^0(\mathbb{R}_+;\mathbb{K}^m)$ such that $(\left. X \right\vert_{[0,T]},\left. u \right\vert_{[0,T]})$ is the result of the above construction for any $T>0$. Then, $(X,u)$ satisfy the estimates (\ref{eq: ISS estimate}-\ref{eq: ISS estimate command}) for all $t \geq 0$.

It remains to show that $(X,u)$ is actually the weak solution associated with $(D,X_0,d_1,d_2)$ of the closed-loop system (\ref{def: boundary control system - closed-loop}). Let $T > 0$ be arbitrarily given. Let $(d_{1,n})_n \in \mathcal{C}^1([0,T];\mathbb{K}^m)^\mathbb{N}$ and $(d_{2,n})_n \in \mathcal{C}^1([0,T];\mathbb{K}^m)^\mathbb{N}$ be approximating sequences converging to $\left. d_1 \right|_{[0,T]}$ and $\left. d_2 \right|_{[0,T]}$, respectively. Thus, the corresponding sequence of mild solutions $(X_{n},u_n)_n$ converges uniformly to $(\left. X \right|_{[0,T]},\left. u \right|_{[0,T]})$. As mild solutions are weak solutions, we obtain that, for any test function $z$ over $[0,T]$ and any $n \geq 1$,
\begin{align*}
& \int_0^T \left< X_n(t), \mathcal{A}_0^* z(t) + \dfrac{\mathrm{d} z}{\mathrm{d} t}(t) \right>_\mathcal{H} \diff t \\
& = - \left< X_0 , z(0) \right>_\mathcal{H} \\
& \phantom{=}\; + \int_0^T \left< B ( u_n(t-D(t)) + d_{1,n}(t) ) , \mathcal{A}_0^* z(t) \right>_\mathcal{H} \diff t \\
& \phantom{=}\; - \int_0^T \left< \mathcal{A} B ( u_n(t-D(t)) + d_{1,n}(t) ) , z(t) \right>_\mathcal{H} \diff t 
\end{align*}
with $\left. u_n \right\vert_{[-D_0-\delta,0]} = 0$ and 
\begin{align*}
u_n(t) & = \varphi(t) \Bigg\{  K Y_n(t) + d_{2,n}(t) \\
& \phantom{=}\; \hspace{1cm} + K \int_{\max(t-D_0,0)}^{t} e^{(t-s-D_0)A_{N_0}} B_{N_0} u_n(s) \diff s \Bigg\}
\end{align*}
for all $t \in [0,T]$, where
\begin{equation*}
Y_n(t) 
= 
\begin{bmatrix}
\left< X_n(t) , \psi_1 \right>_\mathcal{H} &\ldots & \left< X_n(t) , \psi_{N_0} \right>_\mathcal{H}
\end{bmatrix}^\top .
\end{equation*}
Recalling that $\mathcal{A}B$ and $B$ are bounded, we obtain that (\ref{eq: def weak solution}-\ref{eq: def weak solution - control input}) hold for all $t \in [0,T]$ by letting $n \rightarrow + \infty$. As both $T > 0$ and the test function $z$ over $[0,T]$ have been arbitrarily selected, this concludes the proof.
\end{proof}

\section{Conclusion}\label{sec: conclusion}
This paper has investigated the input-to-state stabilization with respect to boundary disturbances of a class of diagonal infinite-dimensional systems via delay boundary control. First, a preliminary lemma regarding the robustness of a constant-delay predictor feedback with respect to uncertain and time-varying input delays has been derived under the form of an ISS estimate with fading memory of the disturbance input. This result was applied to a truncated model capturing the unstable modes of the studied infinite-dimensional system. Finally, this ISS property was extended to the closed-loop infinite-dimensional system, first considering mild solutions and then for weak solutions associated with disturbances exhibiting relaxed regularity assumptions.

\appendix
\section{Proof of Lemma~\ref{lemma: well-posedness closed-loop system}}\label{annex: proof lemma well-posedness}

As $\left. u \right\vert_{[-D_0-\delta,0]} = 0$, (\ref{eq: def mild solution}) is equivalent over $[0,D_0-\delta]$ to 
\begin{align*}
X(t) & = S(t) \{ X_0 - B d_1(0) \} + B d_1(t)  \\ 
& \phantom{=}\; + \int_0^t S(t-s) \{ \mathcal{A}B d_1(s) - B \dot{d}_1(s) \} \diff s , \nonumber
\end{align*}
which is well and uniquely defined as an element of $\mathcal{C}^0([0,D_0-\delta];\mathcal{H})$ with associated control input $u=0 \in \mathcal{C}^1([-D_0-\delta,0];\mathbb{K}^{m})$.

We proceed by induction. Assume that, for a given $n \in \mathbb{N}^*$, there exists a unique pair $(X,u) \in \mathcal{C}^0([0,n(D_0-\delta)];\mathcal{H}) \times \mathcal{C}^1([-D_0-\delta,(n-1)(D_0-\delta)])$ such that (\ref{eq: def mild solution}) holds over $[0,n(D_0-\delta)]$ and (\ref{def: boundary control system - closed-loop - cont input}) holds over $[-D_0-\delta,(n-1)(D_0-\delta)]$. In particular, reproducing the developments reported in Subsection~\ref{subsec: problem and control strategy}, $c_n$ is of class $\mathcal{C}^1$ over $[0,n(D_0-\delta)]$, and thus $Y \in \mathcal{C}^1([0,n(D_0-\delta)],\mathbb{K}^{N_0})$. We show that there exists a unique couple $(\tilde{X},\tilde{u}) \in \mathcal{C}^0([0,(n+1)(D_0-\delta)];\mathcal{H}) \times \mathcal{C}^1([-D_0-\delta,n(D_0-\delta)];\mathbb{K}^m)$ such that 
\begin{align}
\tilde{X}(t) & = S(t) \{ X_0 - B \tilde{v}(0) \} + B \tilde{v}(t) \label{eq: well-posedness induction - def mild solution} \\ 
& \phantom{=}\; + \int_0^t S(t-s) \{ \mathcal{A}B\tilde{v}(s) - B \dot{\tilde{v}}(s) \} \diff s \nonumber
\end{align}
with $\tilde{v}(t) = \tilde{u}(t-D(t)) + d_1(t)$ for all $t \in [0,(n+1)(D_0-\delta)]$ and where the control law is characterized by $\left. \tilde{u} \right\vert_{[-D_0-\delta,0]} = 0$ and, for all $t \in [0,n(D_0-\delta)]$,
\begin{align}
\tilde{u}(t) 
& = \varphi(t) \Bigg\{  K \tilde{Y}(t) + d_2(t) \label{eq: tilde_u} \\
& \phantom{=}\; \hspace{1cm} + K \int_{\max(t-D_0,0)}^{t} e^{(t-s-D_0)A_{N_0}} B_{N_0} \tilde{u}(s) \diff s \Bigg\} \nonumber 
\end{align}
with 
\begin{equation*}
\tilde{Y}(t) 
= \begin{bmatrix}
\left< \tilde{X}(t) , \psi_1 \right>_\mathcal{H} & \ldots & \left< \tilde{X}(t) , \psi_{N_0} \right>_\mathcal{H}
\end{bmatrix}^\top .
\end{equation*}
The induction assumption shows that $\left. \tilde{X} \right\vert_{[0,n(D_0-\delta)]} = X$ and $\left. \tilde{u} \right\vert_{[-D_0-\delta,(n-1)(D_0-\delta)]} = u$. In particular, we have $\tilde{Y}(t) = Y(t)$ for all $t \leq 0 \leq n(D_0-\delta)$. As $t - D(t) \leq n(D_0-\delta)$ for $t \leq (n+1)(D_0-\delta)$, we note that the control input $\tilde{u}$ over the time interval $[0,n(D_0-\delta)]$ is only defined by $X$ over the time interval $[0,n(D_0-\delta)]$ and does not depend on $\tilde{X}$ over $[n(D_0-\delta),(n+1)(D_0-\delta)]$. As $Y \in \mathcal{C}^0([0,n(D_0-\delta)];\mathbb{K}^{N_0})$, we obtain from~\cite[Sec.~III.C]{lhachemi2019feedback} (which is a direct extension of the result reported in~\cite{bresch2018new}, in the configuration $\varphi = 1$, to the case of a continuous function $\varphi$ satisfying $0 \leq \varphi \leq 1$) that the control $\tilde{u}$ given by the implicit equation (\ref{eq: tilde_u}) is well and uniquely defined on $[-D_0-\delta,n(D_0-\delta)]$ as an element of $\mathcal{C}^0([-D_0-\delta,n(D_0-\delta)];\mathbb{K}^m)$ and is such that $\left. \tilde{u} \right\vert_{[-D_0-\delta,(n-1)(D_0-\delta)]} = u$. Introducing
\begin{equation*}
Z(t) = Y(t) + \int_{t-D_0}^{t} e^{(t-D_0-s)A_{N_0}} B_{N_0} \tilde{u}(s) \diff s ,
\end{equation*}
which is such that $Z \in \mathcal{C}^1([0,n(D_0-\delta)];\mathbb{K}^{N_0})$, we can write $\tilde{u}(t) = \varphi(t) K Z(t) + \varphi(t) d_2(t)$ for all $t \in [0,n(D_0-\delta)]$. Thus $\tilde{u}\in\mathcal{C}^1([-D_0-\delta,n(D_0-\delta)])$ and we obtain that (\ref{eq: well-posedness induction - def mild solution}) defines a unique $\tilde{X} \in \mathcal{C}^0([0,(n+1)(D_0-\delta)];\mathcal{H})$. As the obtained $\tilde{X}$ and $\tilde{u}$ are extensions of $X$ and $u$, respectively, this completes the proof by induction.

\section{Mild solutions are weak solutions}\label{annex: mild sol are weak sol}
Let $(X,u) \in \mathcal{C}^0(\mathbb{R}_+;\mathcal{H}) \times \mathcal{C}^1([-D_0-\delta,+\infty);\mathbb{K}^m)$ be a mild solution of (\ref{def: boundary control system - closed-loop}). Then $X$ is given by (\ref{eq: def mild solution}) with $v$ the continuously differentiable function defined by (\ref{eq: boundary input}). For a given $T > 0$, let $z$ be a test function over $[0,T]$, i.e., $z \in \mathcal{C}^0([0,T];D(\mathcal{A}^*_0)) \cap \mathcal{C}^1([0,T];\mathcal{H})$ with $\mathcal{A}_0^* z \in \mathcal{C}^0([0,T];\mathcal{H})$ and $z(T)=0$. We need to show that the system trajectory $X$ satisfies the identity (\ref{eq: def weak solution}). From the basic properties of $C_0$-semigroups and the fundamental theorem of calculus, we infer that
\begin{align*}
& \int_0^T \left< S(t) \{ X_0 - B v(0) \} , \mathcal{A}_0^* z(t) + \dfrac{\mathrm{d} z}{\mathrm{d} t}(t) \right>_\mathcal{H} \diff t \\
& = \int_0^T \left< X_0 - B v(0) , S^*(t) \left\{ \mathcal{A}_0^* z(t) + \dfrac{\mathrm{d} z}{\mathrm{d} t}(t) \right\} \right>_\mathcal{H} \diff t \\
& = \int_0^T \dfrac{\mathrm{d}}{dt} \left( \left< X_0 - B v(0) , S^*(t) z(t) \right>_\mathcal{H} \right) \diff t \\
& = - \left< X_0 - B v(0) , z(0) \right>_\mathcal{H} .
\end{align*}
Using in addition the properties of the Bochner integral and Fubini theorem, we obtain that 
\begin{align*}
& \int_0^T \left< \int_0^t S(t-s) \mathcal{A}Bv(s) \diff s , \mathcal{A}_0^* z(t) + \dfrac{\mathrm{d} z}{\mathrm{d} t}(t) \right>_\mathcal{H} \diff t \\
& = \int_0^T \int_0^t \left< \mathcal{A}Bv(s) , S^*(t-s) \left\{ \mathcal{A}_0^* z(t) + \dfrac{\mathrm{d} z}{\mathrm{d} t}(t) \right\} \right>_\mathcal{H} \diff s \diff t \\
& = \int_0^T \int_s^T \dfrac{\mathrm{d}}{dt} \left( \left< \mathcal{A}Bv(s) , S^*(t-s) z(t) \right>_\mathcal{H} \right) \diff t \diff s \\
& = - \int_0^T \left< \mathcal{A}Bv(s) ,z(s) \right>_\mathcal{H} \diff s .
\end{align*}
Finally, the same approach yields
\begin{align*}
& - \int_0^T \left< \int_0^t S(t-s) B\dot{v}(s) \diff s , \mathcal{A}_0^* z(t) + \dfrac{\mathrm{d} z}{\mathrm{d} t}(t) \right>_\mathcal{H} \diff t \\
& = \int_0^T \left< B\dot{v}(s) ,z(s) \right>_\mathcal{H} \diff s \\
& = - \left< Bv(0) ,z(0) \right>_\mathcal{H} - \int_0^T \left< Bv(s) ,\dfrac{\mathrm{d}z}{\mathrm{d}t}(s) \right>_\mathcal{H} \diff s ,
\end{align*}
where, recalling that $B$ is bounded, the last equality has been derived via an integration by parts. Now, the substitution of the definition of mild solutions (\ref{eq: def mild solution}) into the integral term $\int_0^T \left< X(t), \mathcal{A}_0^* z(t) + \dfrac{\mathrm{d} z}{\mathrm{d} t}(t) \right>_\mathcal{H} \diff t$ and the use of the three latter identities show that the identity (\ref{eq: def weak solution}) is indeed satisfied.

%

\bibliographystyle{cas-model2-names}

\bibliography{cas-refs}

\begin{thebibliography}{38}
\expandafter\ifx\csname natexlab\endcsname\relax\def\natexlab#1{#1}\fi
\providecommand{\url}[1]{\texttt{#1}}
\providecommand{\href}[2]{#2}
\providecommand{\path}[1]{#1}
\providecommand{\DOIprefix}{doi:}
\providecommand{\ArXivprefix}{arXiv:}
\providecommand{\URLprefix}{URL: }
\providecommand{\Pubmedprefix}{pmid:}
\providecommand{\doi}[1]{\href{http://dx.doi.org/#1}{\path{#1}}}
\providecommand{\Pubmed}[1]{\href{pmid:#1}{\path{#1}}}
\providecommand{\bibinfo}[2]{#2}
\ifx\xfnm\relax \def\xfnm[#1]{\unskip,\space#1}\fi
\bibitem[{Artstein(1982)}]{artstein1982linear}
\bibinfo{author}{Artstein, Z.}, \bibinfo{year}{1982}.
\newblock \bibinfo{title}{Linear systems with delayed controls: a reduction}.
\newblock \bibinfo{journal}{IEEE Transactions on Automatic Control}
  \bibinfo{volume}{27}, \bibinfo{pages}{869--879}.
\bibitem[{Ball(1977)}]{ball1977strongly}
\bibinfo{author}{Ball, J.}, \bibinfo{year}{1977}.
\newblock \bibinfo{title}{Strongly continuous semigroups, weak solutions, and
  the variation of constants formula}.
\newblock \bibinfo{journal}{Proceedings of the American Mathematical Society}
  \bibinfo{volume}{63}, \bibinfo{pages}{370--373}.
\bibitem[{Bresch-Pietri et~al.(2018)Bresch-Pietri, Prieur and
  Tr{\'e}lat}]{bresch2018new}
\bibinfo{author}{Bresch-Pietri, D.}, \bibinfo{author}{Prieur, C.},
  \bibinfo{author}{Tr{\'e}lat, E.}, \bibinfo{year}{2018}.
\newblock \bibinfo{title}{New formulation of predictors for finite-dimensional
  linear control systems with input delay}.
\newblock \bibinfo{journal}{Systems \& Control Letters} \bibinfo{volume}{113},
  \bibinfo{pages}{9--16}.
\bibitem[{Bribiesca~Argomedo et~al.(2012)Bribiesca~Argomedo, Witrant and
  Prieur}]{argomedo2012d}
\bibinfo{author}{Bribiesca~Argomedo, F.}, \bibinfo{author}{Witrant, E.},
  \bibinfo{author}{Prieur, C.}, \bibinfo{year}{2012}.
\newblock \bibinfo{title}{{$D^1$}-{I}nput-to-state stability of a time-varying
  nonhomogeneous diffusive equation subject to boundary disturbances}, in:
  \bibinfo{booktitle}{American Control Conference (ACC), 2012},
  \bibinfo{organization}{IEEE}. pp. \bibinfo{pages}{2978--2983}.
\bibitem[{Cai et~al.(2017)Cai, Bekiaris-Liberis and Krstic}]{cai2017input}
\bibinfo{author}{Cai, X.}, \bibinfo{author}{Bekiaris-Liberis, N.},
  \bibinfo{author}{Krstic, M.}, \bibinfo{year}{2017}.
\newblock \bibinfo{title}{Input-to-state stability and inverse optimality of
  linear time-varying-delay predictor feedbacks}.
\newblock \bibinfo{journal}{IEEE Transactions on Automatic Control}
  \bibinfo{volume}{63}, \bibinfo{pages}{233--240}.
\bibitem[{Christensen et~al.(2016)}]{christensen2016introduction}
\bibinfo{author}{Christensen, O.}, et~al., \bibinfo{year}{2016}.
\newblock \bibinfo{title}{An Introduction to Frames and {Riesz} Bases}.
\newblock \bibinfo{publisher}{Springer}.
\bibitem[{Coron and Tr{\'e}lat(2004)}]{coron2004global}
\bibinfo{author}{Coron, J.M.}, \bibinfo{author}{Tr{\'e}lat, E.},
  \bibinfo{year}{2004}.
\newblock \bibinfo{title}{Global steady-state controllability of
  one-dimensional semilinear heat equations}.
\newblock \bibinfo{journal}{SIAM Journal on Control and Optimization}
  \bibinfo{volume}{43}, \bibinfo{pages}{549--569}.
\bibitem[{Coron and Tr{\'e}lat(2006)}]{coron2006global}
\bibinfo{author}{Coron, J.M.}, \bibinfo{author}{Tr{\'e}lat, E.},
  \bibinfo{year}{2006}.
\newblock \bibinfo{title}{Global steady-state stabilization and controllability
  of {1D} semilinear wave equations}.
\newblock \bibinfo{journal}{Communications in Contemporary Mathematics}
  \bibinfo{volume}{8}, \bibinfo{pages}{535--567}.
\bibitem[{Curtain and Zwart(2012)}]{Curtain2012}
\bibinfo{author}{Curtain, R.F.}, \bibinfo{author}{Zwart, H.},
  \bibinfo{year}{2012}.
\newblock \bibinfo{title}{An Introduction to Infinite-Dimensional Linear
  Systems Theory}. volume~\bibinfo{volume}{21}.
\newblock \bibinfo{publisher}{Springer Science \& Business Media}.
\bibitem[{Fridman and Orlov(2009)}]{fridman2009exponential}
\bibinfo{author}{Fridman, E.}, \bibinfo{author}{Orlov, Y.},
  \bibinfo{year}{2009}.
\newblock \bibinfo{title}{Exponential stability of linear distributed parameter
  systems with time-varying delays}.
\newblock \bibinfo{journal}{Automatica} \bibinfo{volume}{45},
  \bibinfo{pages}{194--201}.
\bibitem[{Guzm{\'a}n et~al.(2019)Guzm{\'a}n, Marx and
  Cerpa}]{guzman2019stabilization}
\bibinfo{author}{Guzm{\'a}n, P.}, \bibinfo{author}{Marx, S.},
  \bibinfo{author}{Cerpa, E.}, \bibinfo{year}{2019}.
\newblock \bibinfo{title}{Stabilization of the linear {Kuramoto-Sivashinsky}
  equation with a delayed boundary control}.
\newblock \bibinfo{journal}{{IFAC PapersOnLine}} \bibinfo{volume}{52},
  \bibinfo{pages}{70--75}.
\bibitem[{Jacob et~al.(2018)Jacob, Nabiullin, Partington and
  Schwenninger}]{jacob2018infinite}
\bibinfo{author}{Jacob, B.}, \bibinfo{author}{Nabiullin, R.},
  \bibinfo{author}{Partington, J.R.}, \bibinfo{author}{Schwenninger, F.L.},
  \bibinfo{year}{2018}.
\newblock \bibinfo{title}{Infinite-dimensional input-to-state stability and
  {Orlicz} spaces}.
\newblock \bibinfo{journal}{SIAM Journal on Control and Optimization}
  \bibinfo{volume}{56}, \bibinfo{pages}{868--889}.
\bibitem[{Jacob et~al.(2019)Jacob, Schwenninger and
  Zwart}]{jacob2018continuity}
\bibinfo{author}{Jacob, B.}, \bibinfo{author}{Schwenninger, F.L.},
  \bibinfo{author}{Zwart, H.}, \bibinfo{year}{2019}.
\newblock \bibinfo{title}{On continuity of solutions for parabolic control
  systems and input-to-state stability}.
\newblock \bibinfo{journal}{Journal of differential equations}
  \bibinfo{volume}{266}, \bibinfo{pages}{6284--6306}.
\bibitem[{Karafyllis and Krstic(2013)}]{karafyllis2013delay}
\bibinfo{author}{Karafyllis, I.}, \bibinfo{author}{Krstic, M.},
  \bibinfo{year}{2013}.
\newblock \bibinfo{title}{Delay-robustness of linear predictor feedback without
  restriction on delay rate}.
\newblock \bibinfo{journal}{Automatica} \bibinfo{volume}{49},
  \bibinfo{pages}{1761--1767}.
\bibitem[{Karafyllis and Krstic(2016)}]{karafyllis2016iss}
\bibinfo{author}{Karafyllis, I.}, \bibinfo{author}{Krstic, M.},
  \bibinfo{year}{2016}.
\newblock \bibinfo{title}{{ISS} with respect to boundary disturbances for {1-D}
  parabolic {PDEs}}.
\newblock \bibinfo{journal}{IEEE Transactions on Automatic Control}
  \bibinfo{volume}{61}, \bibinfo{pages}{3712--3724}.
\bibitem[{Karafyllis and Krstic(2017)}]{karafyllis2017iss}
\bibinfo{author}{Karafyllis, I.}, \bibinfo{author}{Krstic, M.},
  \bibinfo{year}{2017}.
\newblock \bibinfo{title}{{ISS} in different norms for {1-D} parabolic {PDEs}
  with boundary disturbances}.
\newblock \bibinfo{journal}{SIAM Journal on Control and Optimization}
  \bibinfo{volume}{55}, \bibinfo{pages}{1716--1751}.
\bibitem[{Karafyllis and Krstic(2019)}]{karafyllis2019preview}
\bibinfo{author}{Karafyllis, I.}, \bibinfo{author}{Krstic, M.},
  \bibinfo{year}{2019}.
\newblock \bibinfo{title}{Input-to-State Stability for {PDEs}}.
\newblock \bibinfo{publisher}{Springer}.
\bibitem[{Krstic(2009)}]{krstic2009control}
\bibinfo{author}{Krstic, M.}, \bibinfo{year}{2009}.
\newblock \bibinfo{title}{Control of an unstable reaction-diffusion {PDE} with
  long input delay}.
\newblock \bibinfo{journal}{Systems \& Control Letters} \bibinfo{volume}{58},
  \bibinfo{pages}{773--782}.
\bibitem[{Lhachemi and Prieur(2021)}]{lhachemi2019feedback}
\bibinfo{author}{Lhachemi, H.}, \bibinfo{author}{Prieur, C.},
  \bibinfo{year}{2021}.
\newblock \bibinfo{title}{Feedback stabilization of a class of diagonal
  infinite-dimensional systems with delay boundary control}.
\newblock \bibinfo{journal}{IEEE Transaction on Automatic Control, accepted, in
  press} .
\bibitem[{Lhachemi et~al.(2019a)Lhachemi, Prieur and Shorten}]{lhachemi2019lmi}
\bibinfo{author}{Lhachemi, H.}, \bibinfo{author}{Prieur, C.},
  \bibinfo{author}{Shorten, R.}, \bibinfo{year}{2019}a.
\newblock \bibinfo{title}{An {LMI} condition for the robustness of
  constant-delay linear predictor feedback with respect to uncertain
  time-varying input delays}.
\newblock \bibinfo{journal}{Automatica} \bibinfo{volume}{109},
  \bibinfo{pages}{108551}.
\bibitem[{Lhachemi et~al.(2019b)Lhachemi, Saussi{\'e}, Zhu and
  Shorten}]{lhachemi2019input}
\bibinfo{author}{Lhachemi, H.}, \bibinfo{author}{Saussi{\'e}, D.},
  \bibinfo{author}{Zhu, G.}, \bibinfo{author}{Shorten, R.},
  \bibinfo{year}{2019}b.
\newblock \bibinfo{title}{Input-to-state stability of a clamped-free damped
  string in the presence of distributed and boundary disturbances}.
\newblock \bibinfo{journal}{IEEE Transactions on Automatic Control, in press} .
\bibitem[{Lhachemi and Shorten(2019a)}]{lhachemi2019boundary1}
\bibinfo{author}{Lhachemi, H.}, \bibinfo{author}{Shorten, R.},
  \bibinfo{year}{2019}a.
\newblock \bibinfo{title}{Boundary feedback stabilization of a
  reaction-diffusion equation with {Robin} boundary conditions and
  state-delay}.
\newblock \bibinfo{journal}{arXiv preprint arXiv:1911.10761} .
\bibitem[{Lhachemi and Shorten(2019b)}]{lhachemi2019boundary2}
\bibinfo{author}{Lhachemi, H.}, \bibinfo{author}{Shorten, R.},
  \bibinfo{year}{2019}b.
\newblock \bibinfo{title}{Boundary input-to-state stabilization of a damped
  {Euler-Bernoulli} beam in the presence of a state-delay}.
\newblock \bibinfo{journal}{arXiv preprint arXiv:1912.01117} .
\bibitem[{Lhachemi and Shorten(2019c)}]{lhachemi2018iss}
\bibinfo{author}{Lhachemi, H.}, \bibinfo{author}{Shorten, R.},
  \bibinfo{year}{2019}c.
\newblock \bibinfo{title}{{ISS} property with respect to boundary disturbances
  for a class of {Riesz-spectral} boundary control systems}.
\newblock \bibinfo{journal}{Automatica} \bibinfo{volume}{109},
  \bibinfo{pages}{108504}.
\bibitem[{Lhachemi et~al.(2019c)Lhachemi, Shorten and
  Prieur}]{lhachemi2019control}
\bibinfo{author}{Lhachemi, H.}, \bibinfo{author}{Shorten, R.},
  \bibinfo{author}{Prieur, C.}, \bibinfo{year}{2019}c.
\newblock \bibinfo{title}{Control law realification for the feedback
  stabilization of a class of diagonal infinite-dimensional systems with delay
  boundary control}.
\newblock \bibinfo{journal}{IEEE Control Systems Letters} \bibinfo{volume}{3},
  \bibinfo{pages}{930--935}.
\bibitem[{Mironchenko et~al.(2019)Mironchenko, Karafyllis and
  Krstic}]{mironchenko2019monotonicity}
\bibinfo{author}{Mironchenko, A.}, \bibinfo{author}{Karafyllis, I.},
  \bibinfo{author}{Krstic, M.}, \bibinfo{year}{2019}.
\newblock \bibinfo{title}{Monotonicity methods for input-to-state stability of
  nonlinear parabolic {PDEs} with boundary disturbances}.
\newblock \bibinfo{journal}{SIAM Journal on Control and Optimization}
  \bibinfo{volume}{57}, \bibinfo{pages}{510--532}.
\bibitem[{Mironchenko and Prieur(2019)}]{mironchenko2019input}
\bibinfo{author}{Mironchenko, A.}, \bibinfo{author}{Prieur, C.},
  \bibinfo{year}{2019}.
\newblock \bibinfo{title}{Input-to-state stability of infinite-dimensional
  systems: recent results and open questions}.
\newblock \bibinfo{journal}{arXiv preprint arXiv:1910.01714} .
\bibitem[{Mironchenko and Wirth(2016)}]{mironchenko2016restatements}
\bibinfo{author}{Mironchenko, A.}, \bibinfo{author}{Wirth, F.},
  \bibinfo{year}{2016}.
\newblock \bibinfo{title}{Restatements of input-to-state stability in infinite
  dimensions: what goes wrong}, in: \bibinfo{booktitle}{Proc. of 22th
  International Symposium on Mathematical Theory of Systems and Networks (MTNS
  2016)}, pp. \bibinfo{pages}{667--674}.
\bibitem[{Mironchenko and Wirth(2018)}]{mironchenko2017characterizations}
\bibinfo{author}{Mironchenko, A.}, \bibinfo{author}{Wirth, F.},
  \bibinfo{year}{2018}.
\newblock \bibinfo{title}{Characterizations of input-to-state stability for
  infinite-dimensional systems}.
\newblock \bibinfo{journal}{IEEE Transactions on Automatic Control}
  \bibinfo{volume}{63}, \bibinfo{pages}{1692--1707}.
\bibitem[{Prieur and Tr{\'e}lat(2019)}]{prieur2018feedback}
\bibinfo{author}{Prieur, C.}, \bibinfo{author}{Tr{\'e}lat, E.},
  \bibinfo{year}{2019}.
\newblock \bibinfo{title}{Feedback stabilization of a {1D} linear
  reaction-diffusion equation with delay boundary control}.
\newblock \bibinfo{journal}{IEEE Transactions on Automatic Control}
  \bibinfo{volume}{64}, \bibinfo{pages}{1415--1425}.
\bibitem[{Richard(2003)}]{richard2003time}
\bibinfo{author}{Richard, J.P.}, \bibinfo{year}{2003}.
\newblock \bibinfo{title}{Time-delay systems: an overview of some recent
  advances and open problems}.
\newblock \bibinfo{journal}{Automatica} \bibinfo{volume}{39},
  \bibinfo{pages}{1667--1694}.
\bibitem[{Russell(1978)}]{russell1978controllability}
\bibinfo{author}{Russell, D.L.}, \bibinfo{year}{1978}.
\newblock \bibinfo{title}{Controllability and stabilizability theory for linear
  partial differential equations: recent progress and open questions}.
\newblock \bibinfo{journal}{{SIAM} Review} \bibinfo{volume}{20},
  \bibinfo{pages}{639--739}.
\bibitem[{Selivanov and Fridman(2016)}]{selivanov2016observer}
\bibinfo{author}{Selivanov, A.}, \bibinfo{author}{Fridman, E.},
  \bibinfo{year}{2016}.
\newblock \bibinfo{title}{Observer-based input-to-state stabilization of
  networked control systems with large uncertain delays}.
\newblock \bibinfo{journal}{Automatica} \bibinfo{volume}{74},
  \bibinfo{pages}{63--70}.
\bibitem[{Solomon and Fridman(2015)}]{solomon2015stability}
\bibinfo{author}{Solomon, O.}, \bibinfo{author}{Fridman, E.},
  \bibinfo{year}{2015}.
\newblock \bibinfo{title}{Stability and passivity analysis of semilinear
  diffusion {PDEs} with time-delays}.
\newblock \bibinfo{journal}{International Journal of Control}
  \bibinfo{volume}{88}, \bibinfo{pages}{180--192}.
\bibitem[{Sontag(2008)}]{sontag2008input}
\bibinfo{author}{Sontag, E.D.}, \bibinfo{year}{2008}.
\newblock \bibinfo{title}{Input to state stability: Basic concepts and
  results}, in: \bibinfo{booktitle}{Nonlinear and optimal control theory}.
  \bibinfo{publisher}{Springer}, pp. \bibinfo{pages}{163--220}.
\bibitem[{Sontag(Apr. 1989)}]{sontag1989smooth}
\bibinfo{author}{Sontag, E.D.}, \bibinfo{year}{Apr. 1989}.
\newblock \bibinfo{title}{Smooth stabilization implies coprime factorization}.
\newblock \bibinfo{journal}{IEEE Transactions on Automatic Control}
  \bibinfo{volume}{34}, \bibinfo{pages}{435--443}.
\bibitem[{Zheng and Zhu(2018a)}]{zheng2018giorgi}
\bibinfo{author}{Zheng, J.}, \bibinfo{author}{Zhu, G.}, \bibinfo{year}{2018}a.
\newblock \bibinfo{title}{A {De Giorgi} iteration-based approach for the
  establishment of {ISS} properties for {Burgers’} equation with boundary and
  in-domain disturbances}.
\newblock \bibinfo{journal}{IEEE Transactions on Automatic Control}
  \bibinfo{volume}{64}, \bibinfo{pages}{3476--3483}.
\bibitem[{Zheng and Zhu(2018b)}]{zheng2017input}
\bibinfo{author}{Zheng, J.}, \bibinfo{author}{Zhu, G.}, \bibinfo{year}{2018}b.
\newblock \bibinfo{title}{Input-to-state stability with respect to boundary
  disturbances for a class of semi-linear parabolic equations}.
\newblock \bibinfo{journal}{Automatica} \bibinfo{volume}{97},
  \bibinfo{pages}{271--277}.

\end{thebibliography}


%
%

\end{document}